\documentclass[11pt]{amsart}

\usepackage{amsmath,amssymb,amsxtra,epsf,amscd,graphics,color,array,ulem,etex, colortbl}
\usepackage{mathrsfs}
\usepackage{pstricks}
\usepackage[latin1]{inputenc}
\usepackage[T1]{fontenc}
\usepackage[frenchb, english]{babel}

\usepackage{amsfonts}

\usepackage{ytableau}

\usepackage{enumerate,fancybox,array,calc}
\usepackage{epsf}
\usepackage{epsfig}
\usepackage{verbatim,ifthen, fancyvrb,xkeyval,xcolor,epic,eepic,rotating}

\usepackage{hhline}
\usepackage{tikz}
\usetikzlibrary{patterns}

\newcommand{\g}{\mathfrak{g}}

\def\p{\mathfrak p}
\def\s{\mathfrak s}
\def\h{\mathfrak h}

\def\a{\mathfrak a}
\def\m{\mathfrak m}

\def\n{\mathfrak n}

\def\q{\mathfrak q}

\newtheorem*{thm}{Theorem}

\newtheorem*{prop}{Proposition}
\newtheorem*{lm}{Lemma}

\newtheorem*{Rq}{Remark}

\title[Symmetric semi-invariants for some  In\"on\"u-Wigner contractions]{Symmetric semi-invariants for some  In\"on\"u-Wigner contractions.}

\author{Florence Fauquant-Millet}

\address{Univ Lyon,
UJM Saint-Etienne\\
CNRS UMR 5208\\
Institut Camille Jordan\\
F- 42023 Saint-Etienne\\
France}
\email{florence.millet@univ-st-etienne.fr}

\usepackage{hyperref}
\usepackage{amsfonts,amssymb,wasysym}

\begin{document}

\begin{abstract}

Let $\p$ be a proper parabolic subalgebra of a simple Lie algebra $\g$. Writing $\p=\mathfrak r\oplus\mathfrak m$ with $\mathfrak r$ being the Levi factor of $\p$ and $\m$  the nilpotent radical of $\p$, we may consider the semi-direct product $\tilde\p=\mathfrak r\ltimes(\m)^a$, where $(\m)^a$ is an abelian ideal of $\tilde\p$, isomorphic to $\m$ as an $\mathfrak r$-module. 
 Then $\tilde\p$ is a Lie algebra, which  is a special case of  In\"on\" u-Wigner contraction and may be considered as a degeneration of the parabolic subalgebra $\p$. Let $S(\tilde\p)$ be the symmetric algebra of $\tilde\p$ (it is equal to the symmetric algebra $S(\p)$ of $\p$) and consider the algebra of semi-invariants $Sy(\tilde\p)\subset S(\tilde\p)$ under the adjoint action of $\tilde\p$. Using what we call a generalized PBW filtration on a highest weight irreducible representation $V(\lambda)$ of $\g$,  induced by the standard degree filtration on $U(\m^-)$ (where $\m^-$ is the nilpotent radical of the opposite subalgebra $\p^-$ of $\p$) one obtains a lower bound for the formal character of  the algebra $Sy(\tilde\p)$, when the latter is well defined.

\end{abstract}

\maketitle



\section{Introduction.}\label{Intro}

The base field $\Bbbk$ is algebraically closed of characteristic zero.

\subsection{The aim of the paper.}\label{Introbounds}

Let $\g$ be a simple Lie algebra over $\Bbbk$ and fix a Cartan subalgebra $\h$ of $\g$. Then choose a set $\pi$ of simple roots for $(\g,\,\h)$ and denote by $\mathfrak b$ the Borel subalgebra of $\g$ associated with it. Let $\p\supset\mathfrak b$ be a proper parabolic subalgebra of $\g$. Denote by $\n$, resp. $\n^-$, the maximal nilpotent subalgebra of $\g$ generated by all positive, resp. negative, root vectors, so that $\g=\n^-\oplus\h\oplus\n$ and $\mathfrak b=\h\oplus\n$. Let $\mathfrak r$ denote the Levi factor of $\p$ (so that $\mathfrak r$ is a reductive Lie algebra) and $\m$ the nilpotent radical of $\p$. Then one has that $\p=\mathfrak r\oplus\m$. Now consider the semi-direct product $\tilde\p=\mathfrak r\ltimes(\m)^a$  where $(\m)^a$ is isomorphic to $\m$ as  an $\mathfrak r$-module, the superscript $a$ meaning hat $(\m)^a$ is an abelian ideal of $\tilde\p$. The semi-direct product $\tilde\p$ is still a Lie algebra which may be viewed as a {\it degeneration} of the parabolic subalgebra $\p$. It is called an In\"on\" u-Wigner contraction, or a one-parameter contraction of $\p$ (see \cite[Sect. 4] {Y1}). Denoting by $\a'$ the derived subalgebra of any Lie algebra $\a$, one has that $\tilde\p'=\mathfrak r'\ltimes(\m)^a$.  

In this paper we are interested in the algebra $Sy(\tilde\p)$ of symmetric semi-invariants in the symmetric algebra $S(\tilde\p)$ of $\tilde\p$ under the adjoint action of $\tilde\p$, which is also equal to the algebra $S(\tilde\p)^{\tilde\p'}$ of symmetric invariants under the adjoint action of $\tilde\p'$. In some cases (especially when $\p$ is a maximal parabolic subalgebra), we have that $Sy(\tilde\p)=S(\tilde\p')^{\tilde\p'}$. For the natural Poisson structure on $S(\tilde\p)$, the algebra $Sy(\tilde\p)$ is also equal to the Poisson semicentre of $S(\tilde\p)$. Roughly speaking, we may view the algebra $Sy(\tilde\p)$ as a {\it degeneration} of the algebra of symmetric semi-invariants $Sy(\p)=S(\p)^{\p'}$ in $S(\p)$. The aim of  the present paper is to construct a lower bound for the formal character of $Sy(\tilde\p)$ (when the latter is well defined), which will be shown  to be equal to the lower bound for the formal character of $Sy(\p)$, as computed in \cite[Sect. 6]{FJ2} (see also \cite[Prop. 3.1]{FJ1}). 
We hope then, when $\p$ is a maximal parabolic subalgebra of $\g$, to compare this lower bound with an upper bound, given by an adapted pair of $\tilde\p'$ and to show that both bounds coincide : this will imply that in this case, the algebra $Sy(\tilde\p)$ is a polynomial algebra, for which we can give the number of algebraically independent generators, their weight and degree.

Similar semi-direct products were studied extensively by Panyushev and Yakimova in \cite{P}, \cite{PPY1}, \cite{PPY2}, \cite{PPY3}, \cite{Y1}, \cite{Y2}. In particular these authors studied the polynomiality of the algebra of symmetric invariants $S(\q)^{\q}$ for semi-direct products $\q=\a\ltimes V$ where $\a$ is a {\bf simple} Lie algebra  and $V$ is a finite-dimensional representation of $\a$. For any type of simple Lie algebra $\a$ (except in type ${\rm A}$ where their study is partial) they established a list of all representations $V$, up to isomorphism,  of $\a$ for which the algebra $S(\q)^{\q}$ is polynomial and they gave the number of algebraically independent generators.
Observe that in our paper we deal with a semi-direct product $\q=\a\ltimes V$ with $\a=\mathfrak r'$ being {\bf semisimple} (and not necessarily simple in general) and $V=\m$.
Note  that it is shown in \cite[Th. 1.1]{PPY1} that the bi-homogeneous components of highest degree relative to $\m$ of homogeneous elements in $S(\p')^{\p'}$  lie  in $S(\tilde\p')^{\tilde\p'}=S(\p')^{\tilde\p'}$. Moreover by \cite[Th. 3.8]{Y0},  if $S(\p')^{\p'}$ is polynomial, generated by a set of algebraically independent homogeneous generators satisfying further conditions, one may know whether $S(\p')^{\tilde\p'}$ is also polynomial : it happens
 if and only if the sum of their degrees relative to $\m$ is equal to $\dim\m$. Unfortunately, even when the degree is known for each generator of $S(\p')^{\p'}$, it does not seem to be easy to compute its degree relative to $\m$.

\subsection{The method}
The method we use in this paper is completely different from this of Panyushev and Yakimova. Our method is partly inspired by this used in \cite{FJ1}, \cite{FJ2} to study the polynomiality of the algebra of symmetric semi-invariants $Sy(\p)$. The study of the latter algebra will be called {\it the nondegenerate case}, while we will call the study of $Sy(\tilde\p)$ {\it the degenerate case}.

Our aim is  to construct a lower bound for the algebra $Sy(\tilde\p)$ of semi-invariants. This bound will be given by the algebra of matrix coefficients on some degenerate module built from  the irreducible highest weight $\g$-module $V(\lambda)$ of highest weight $\lambda$, for $\lambda\in P^+(\pi)$, where $P^+(\pi)$ is the set of dominant integral weights of $(\g,\,\h)$. 

Let us describe our method and our main result more precisely.
\begin{itemize}

\item In subsections \ref{PBW} and \ref{descr} we fix
$\lambda\in P^+(\pi)$ and 
 denote by $\p^-=\mathfrak r\oplus\m^-\supset\mathfrak b^-=\h\oplus\n^-$  the opposite parabolic subalgebra of $\p$, where $\m^-$ is the nilpotent radical of $\p^-$ and by $\tilde\p^-=\mathfrak r\ltimes(\m^-)^a$ the one-parameter contraction of $\p^-$. Then, inspired by the construction in \cite{FFL}, we define what we call a {\it generalized PBW filtration} $(\mathscr F_k(V(\lambda))_{k\in\mathbb N}$ on $V(\lambda)$, which is an increasing and exhaustive filtration on $V(\lambda)$,
 induced by the canonical (or standard degree) filtration $(U_k(\m^-))_{k\in\mathbb N}$ on the enveloping algebra $U(\m^-)$ of $\m^-$.

The associated graded space, that we call the {\it degenerate highest weight module associated with $\lambda$}, is denoted by $$\widetilde{V}(\lambda):=gr_{\mathscr F}(V(\lambda))=\bigoplus_{k\in\mathbb N}gr_k(V(\lambda))$$
where $gr_k(V(\lambda))=\frac{\mathscr F_k(V(\lambda))}{\mathscr F_{k-1}(V(\lambda))}$  for all $k\in\mathbb N$ with $\mathscr F_{-1}(V(\lambda)):=\{0\}$. 
 
If $v_{\lambda}$ is a nonzero vector of highest weight $\lambda$ in $V(\lambda)$, we denote by $V'(\lambda)$ the irreducible $U(\mathfrak r)$-submodule of $V(\lambda)$ generated by $v_{\lambda}$ and  by $\widetilde{V'}(\lambda)\subset\widetilde{V}(\lambda)$
 the canonical image of $V'(\lambda)$ in $\widetilde{V}(\lambda)$.
We will observe that, as $U(\mathfrak r)$-modules, we have $\widetilde{V}(\lambda)\simeq V(\lambda)$. 
Set $\tilde{v}_{\lambda}$ the canonical image of $v_{\lambda}$ in $\widetilde{V}(\lambda)$. We define a  left $U(\tilde\p^-)$-module structure on $\widetilde{V}(\lambda)$,
 for which we have that $\widetilde{V'}(\lambda)=U(\mathfrak r).\tilde{v}_{\lambda}$ and
that $$\widetilde{V}(\lambda)=U(\tilde\p^-).\tilde{v}_{\lambda}=S(\m^-).\widetilde{V'}(\lambda)=U(\tilde\p^-).\widetilde{V'}(\lambda).$$  

\item In subsections \ref{smash}, \ref{coadjointaction},  \ref{adjaction},
denoting by $T(\m)$ the tensor algebra of $\m$, we consider the associative algebra $A=T(\m)\string # U(\mathfrak r)$, which is the Hopf smash product of the left $U(\mathfrak r)$-algebra $T(\m)$ by the Hopf algebra $U(\mathfrak r)$, as defined for example in \cite[1.1.8]{J2}. As $T(\m)$ is also equipped with a coproduct, we obtain that this smash product $A$ also inherits a structure of a bialgebra. 
  
We then consider the coadjoint action, which we denote by $ad^*$, of $U(\tilde\p)$ on $\p^-\simeq\tilde\p^*$ (as vector spaces). Then  $ad^*$ extends uniquely  by derivation to an action of $U(\tilde\p)$ on $S(\p^-)$. From this action $ad^*$, we define what we call a {\it generalized adjoint action} $ad^{**}$ of $A$ on $U(\tilde\p^-)$, which coincides with the adjoint action  on $U(\tilde\p^-)$, when restricted to $U(\mathfrak r)$.

\item In subsections \ref{def} and \ref{C}, we consider spaces of matrix coefficients.

For $\lambda\in P^+(\pi)$, we set $\tilde v_{w_0\lambda}$ the canonical image in $\widetilde{V}(\lambda)$ of a chosen nonzero lowest weight vector in $V(\lambda)$ and by $\widetilde{V''}(\lambda)$ the $U(\mathfrak r)$-submodule of $\widetilde{V}(\lambda)$ generated by  $\tilde v_{w_0\lambda}$. We denote by $\widetilde{V}(\lambda)^*$  the dual space of $\widetilde{V}(\lambda)$. For all $\xi\in \widetilde{V}(\lambda)^*$ and $v\in\widetilde{V'}(\lambda)$, the matrix coefficient $c_{\xi,\,v}\in U(\tilde\p^-)^*$ is defined by :
$$c_{\xi,\,v}(u)=\xi( u.\,v)\;\;\hbox{\rm for\;all}\; u\in U(\tilde\p^-).$$

Then we define 
$\widetilde{C}_{\mathfrak p}(\lambda)$ to be the subspace of $U(\tilde\p^-)^*$ generated by
$$\{c_{\xi,\,v}\mid \xi\in\widetilde{V}(\lambda)^*,\,v\in \widetilde{V'}(\lambda)\}$$ and ${\widetilde C}_{\mathfrak r}(\lambda)$ to be the subspace of $\widetilde C_{\p}(\lambda)$ generated by $$\{c_{\xi,\,v}\mid \xi \in\widetilde{V^{''}}(\lambda)^*,\,v\in \widetilde{V}'(\lambda)\}.$$
We set $\widetilde{C}_{\p}=\sum_{\lambda\in P^+(\pi)}\widetilde{C}_{\mathfrak p}(\lambda)$ and ${\widetilde C}_{\mathfrak r}=\sum_{\lambda\in P^+(\pi)}{\widetilde C}_{\mathfrak r}(\lambda)$. We show that these are direct sums
and  that ${\widetilde C}_{\mathfrak r}$ is a subalgebra of $U(\tilde\p^-)^*$.

\item In subsection \ref{dualrepresentation}, we consider
the dual representation of $ad^{**}$, which defines a left $A$-module structure on $U(\tilde\p^-)^*$. When restricted to $U(\mathfrak r)$, the dual representation of $ad^{**}$ defines a left $U(\mathfrak r)$-module structure on every $\widetilde{C}_{\mathfrak r}(\lambda)$, $\lambda\in P^+(\pi)$, and then on $\widetilde C_{\mathfrak r}$, which coincides with the coadjoint representation.

\item In subsections \ref{polynomial} and \ref{filtrationC}, for all $\lambda\in P^+(\pi)$, we
denote by ${\widetilde C}_{\mathfrak r}(\lambda)^{U(\mathfrak r')}$, resp. ${\widetilde C}_{\mathfrak r}^{U(\mathfrak r')}$ the vector space, resp. the algebra, of invariants in ${\widetilde C}_{\mathfrak r}(\lambda)$, resp. in ${\widetilde C}_{\mathfrak r}$, by the coadjoint representation of $U(\mathfrak r')$. We have that
$${\widetilde C}_{\mathfrak r}^{U(\mathfrak r')}=\bigoplus_{\lambda\in P^+(\pi)}{\widetilde C}_{\mathfrak r}(\lambda)^{U(\mathfrak r')}.$$
Denote by $\pi'\subset\pi$  the subset of simple roots of $(\g,\h)$ associated with the parabolic subalgebra $\p$, set $\h_{\pi'}=\h\cap\p'$, and denote by $(\,\,,\,)$ the non degenerate symmetric bilinear form  on $\h^*\times\h^*$ induced by the Killing form on $\g$. 
 Since, for all $\lambda\in P^+(\pi)$, $\widetilde{V'}(\lambda)$ is an irreducible $U(\mathfrak r)$-module, the Jacobson density theorem implies that the $U(\mathfrak r)$-module $\widetilde{C}_{\mathfrak r}(\lambda)$ is isomorphic to the $U(\mathfrak r)$-module $\widetilde{V''}(\lambda)^*\otimes\widetilde{V'}(\lambda)$ where the latter is endowed with the diagonal action of $U(\mathfrak r)$. 
  It follows that, for all $\lambda\in P^+(\pi)$, ${\widetilde C}_{\mathfrak r}(\lambda)^{U(\mathfrak r')}$ is of dimension less or equal to one, and equal to one if and only if 
 $$(w_0'\lambda-w_0\lambda,\,\pi')=0$$
 where $w_0'$, resp. $w_0$, is the longest element in the Weyl group of $(\mathfrak r',\,\h_{\pi'})$, resp. of $(\g,\,\h)$.  As a consequence  we show (as in \cite[prop. 3.1]{FJ1}) that ${\widetilde C}_{\mathfrak r}^{U(\mathfrak r')}$ is a polynomial algebra, for which we can compute the weight of each  vector of a set of algebraically independent  generators.

\item In subsections \ref{Kf}, \ref{dualS}, \ref{Km} and \ref{isomtildep},   inspired by \cite[6.1]{FJ2}, one defines on the algebra $U(\tilde\p^-)^*$ what we call {\it the generalized Kostant filtration} $(\mathscr F_K^k(U(\tilde\p^-)^*))_{k\in\mathbb N}$,  which is a decreasing, exhaustive and separated ring filtration. This filtration is invariant under the action of $A$ given by the dual representation of $ad^{**}$. 

One denotes by $gr_K(U(\tilde\p^-)^*)=\bigoplus_{k\in\mathbb N} gr_K^k(U(\tilde\p^-)^*)$ the graded algebra associated with this filtration where, for all $k\in\mathbb N$, $$gr_K^k(U(\tilde\p^-)^*)=\frac{\mathscr F_K^k(U(\tilde\p^-)^*)}{\mathscr F_K^{k+1}(U(\tilde\p^-)^*)}.$$
The dual representation of $ad^{**}$ induces a left action of $A$ on this graded algebra and one checks that, for all $x\in\m$, for all $f\in\widetilde{C}_{\mathfrak r}\cap \mathscr F_K^k(U(\tilde\p^-)^*)$, one has for this action $$x.f\in\mathscr F_K^{k+1}(U(\tilde\p^-)^*)$$ that is, that 
\noindent\begin{equation}x.gr_K^k(f)=0\label{annulm}\tag{$\diamond$}\end{equation}  where $gr_K^k(f)$ denotes the canonical image of $f$ in $gr_K(U(\tilde\p^-)^*)$.

Then  for all $k\in\mathbb N$ and all vector space $V$, denoting by $S_k(V)$  the vector subspace of the symmetric algebra $S(V)$ of $V$ formed by all homogeneous polynomials of degree $k$, 
one defines a morphism  $\psi_k: gr_K^k(U(\tilde\p^-)^*)\longrightarrow S_k(\p^-)^*$.
 It is easily checked that actually $\psi_k$ is an isomorphism of left $U(\tilde\p)$-modules, where the left structure on $gr_K^k(U(\tilde\p^-)^*)$ is induced by the dual representation of $ad^{**}$ and where the left structure on $S_k(\p^-)^*$ is given by the dual representation of $ad^*$. With this structure, it is easily checked that
  $S_k(\p^-)^*$ is isomorphic to the $U(\tilde\p)$-module $S_k(\tilde\p)=S_k(\p)$ where  the action of $\tilde\p$ is the adjoint action which extends by derivation the Lie bracket in $\tilde\p$.
Thus we obtain an isomorphism of $U(\tilde\p)$-modules and of algebras from $gr_K(U(\tilde\p^-)^*)$ to $S(\tilde\p)$.

\item 

Denote  by $gr_K({\widetilde C}_{\mathfrak r}^{U(\mathfrak r')})$  the graded algebra associated with the induced generalized Kostant filtration on ${\widetilde C}_{\mathfrak r}^{U(\mathfrak r')}$.  The former may be viewed as a subalgebra of $gr_K(U(\tilde\p^-)^*)$, which by equation (\ref{annulm}) is invariant under the  action of $U(\tilde\p')$ induced by the action of $A$ on $U(\tilde\p^-)^*$ given by the dual representation of $ad^{**}$. Finally one can establish the main result of our paper (see subsection \ref{inj}). 
\begin{thm}
There is an injection of algebras and of $U(\h)$-modules from $gr_K({\widetilde C}_{\mathfrak r}^{U(\mathfrak r')})$ into the Poisson semicentre $Sy(\tilde\p)=S(\tilde\p)^{\tilde\p'}$. This implies a lower bound for the formal character of $Sy(\tilde\p)$, when the latter is well defined.
\end{thm}

\end{itemize}

\section{Notation.}

\subsection{General notation}\label{genotation}

Let $\g$ be a simple Lie algebra over $\Bbbk$, $\h$ be a Cartan subalgebra of $\g$ and choose a set $\pi$ of simple roots for $(\g,\,\h)$. Denote by $\Delta^\pm$  the set of positive, resp. negative, roots of $(\g,\,\h)$ with respect to $\pi$ and $\Delta=\Delta^+\sqcup\Delta^-$ the set of roots of $(\g,\,\h)$. 
Denote by $[\,,\,]$ the Lie bracket in $\g$ and by $\langle\,,\,\rangle$ the natural duality between $\h$ and $\h^*$. Then for all root $\alpha\in\Delta$, set $\g_{\alpha}=\{x\in\g\mid\forall h\in\h,\,[h,\,x]=\langle h,\,\alpha\rangle x\}$ and fix a nonzero root vector $x_{\alpha}$ in $\g_{\alpha}$.

Denote by $\n=\bigoplus_{\alpha\in\Delta^+}\g_{\alpha}$, resp. $\n^-=\bigoplus_{\alpha\in\Delta^-}\g_{\alpha}$, the maximal nilpotent subalgebra of $\g$ generated by positive, resp. negative, root vectors,  so that $\g=\n\oplus\h\oplus\n^-$. Let $\mathfrak b=\n\oplus\h$ be the Borel subalgebra of $\g$.

For each subset $\pi'$ of $\pi$, we denote by $\Delta^\pm_{\pi'}$ the subset of $\Delta^\pm$ generated by $\pi'$ that is,
$\Delta^\pm_{\pi'}=(\pm\mathbb N\pi')\cap\Delta^\pm$. Set $\n_{\pi'}=\bigoplus_{\alpha\in\Delta^+_{\pi'}}\g_{\alpha}$, resp. $\n^-_{\pi'}=\bigoplus_{\alpha\in\Delta^-_{\pi'}}\g_{\alpha}$. Then the (standard) parabolic subalgebra $\p\supset\mathfrak b$ of $\g$ associated with $\pi'$ is 
$$\p=\n\oplus\h\oplus\n^-_{\pi'}.$$

The Levi factor $\mathfrak r$ of $\p$ is $$\mathfrak r=\n_{\pi'}\oplus\h\oplus\n^-_{\pi'}$$ and its derived subalgebra (which is semisimple) is
$\mathfrak r'=\n_{\pi'}\oplus\h_{\pi'}\oplus\n^-_{\pi'}$, where $\h_{\pi'}=\h\cap \p'$, with $\p'=[\p,\,\p]$ being the derived subalgebra of $\p$. If for all $\alpha\in\pi$, $\alpha\check\null$ denotes the coroot associated with $\alpha$, we have that $\h_{\pi'}$ is the $\Bbbk$-vector space generated by the coroots $\alpha\check\null$, with $\alpha\in\pi'$.

The longest element of the Weyl group $W$, resp. $W'$, of $(\g,\,\h)$, resp. of $(\mathfrak r',\,\h_{\pi'})$, is denoted by $w_0$, resp. $w_0'$. 

Set $\h^{\pi\setminus\pi'}=\{h\in\h\mid \langle h,\,\pi'\rangle=0\}$ so that $\h=\h_{\pi'}\oplus\h^{\pi\setminus\pi'}$.
Denote by $\m$ the nilpotent radical of $\p$, so that $\p=\mathfrak r\oplus\m$. We have that $\n=\n_{\pi'}\oplus\m$ and that $\m=\bigoplus_{\alpha\in\Delta^+\setminus\Delta^+_{\pi'}}\g_{\alpha}$.
The opposite subalgebra $\p^-$ of $\p$ is the parabolic subalgebra of $\g$ defined by
$$\p^-=\n^-\oplus\h\oplus\n_{\pi'}.$$ We denote by $\m^-$ the nilpotent radical of $\p^-$ (so that $\p^-=\mathfrak r\oplus\m^-$). The Killing form $K$ on $\g\times\g$ induces an isomorphism between the dual space $\p^*$ of $\p$ and the vector space $\p^-$, since $K$ is non degenerate on $\p\times\p^-$. Moreover since $K$ is also non degenerate on $\h\times\h$, it induces a non degenerate symmetric bilinear form $(\,\,,\,)$ on $\h^*\times\h^*$ which is invariant under the action of $W$ (see for instance \cite[5.2.2]{FJ3}).

For all $\alpha\in\pi$, resp. $\alpha\in\pi'$, let $\varpi_{\alpha}$, resp. $\varpi'_{\alpha}$, be the fundamental weight associated with $\alpha$ with respect to $(\g,\,\h)$, resp. with respect to $(\mathfrak r',\,\h_{\pi'})$. Then $P(\pi)=\sum_{\alpha\in\pi}\mathbb Z\varpi_{\alpha}$, resp. $P(\pi')=\sum_{\alpha\in\pi'}\mathbb Z\varpi'_{\alpha}$, is the weight lattice of $(\g,\,\h)$, resp. $(\mathfrak r',\,\h_{\pi'})$. Moreover $P^+(\pi)=\sum_{\alpha\in\pi}\mathbb N\varpi_{\alpha}$, resp. $P^+(\pi')=\sum_{\alpha\in\pi'}\mathbb N\varpi'_{\alpha}$, is the set of dominant integral weights of $(\g,\,\h)$, resp. $(\mathfrak r',\,\h_{\pi'})$.
By \cite[2.5]{FJ2}, there exists some positive integer $r$ such that

\begin{equation}P(\pi)\subset P(\pi')\oplus\frac{1}{r}\sum_{\alpha\in\pi\setminus\pi'}\mathbb Z\varpi_{\alpha}\label{projP}\end{equation}
and for all $\alpha\in\pi'$, the projection of  $\varpi_{\alpha}$ in $P(\pi')$ with respect to this decomposition (\ref{projP}) is $\varpi'_{\alpha}$.
For $\lambda=\sum_{\alpha\in\pi}m_{\alpha}\varpi_{\alpha}\in P(\pi)$ ($m_{\alpha}\in\mathbb Z$ for each $\alpha\in\pi$), we denote by $\lambda'=\sum_{\alpha\in\pi'}m_{\alpha}\varpi'_{\alpha}$ its projection in $P(\pi')$ with respect to the decomposition (\ref{projP}).

For any finite-dimensional Lie algebra $\a$, we denote by $U(\a)$ its universal enveloping algebra and by $S(\a)$ its symmetric algebra, which may be viewed as the (commutative) graded algebra associated with the canonical filtration $(U_k(\a))_{k\in\mathbb N}$ on $U(\a)$ (see \cite[2.3]{D}). We may also identify $S(\a)$ with the algebra $\Bbbk[\a^*]$ of polynomial functions on the dual space $\a^*$ of $\a$. For all $k\in\mathbb N$, we denote by $S_k(\a)$ the vector subspace of $S(\a)$ formed by all homogeneous polynomials of degree $k$.

 For all $\lambda\in P^+(\pi)$,  the irreducible highest weight $\g$-module of highest weight $\lambda$ (which is obtained by quotienting  the corresponding Verma module by its largest proper sub-$\g$-module, as defined for example in \cite[7.1.11]{D}) is denoted by $V(\lambda)$ : recall (\cite[7.2.6]{D}) that this is a finite-dimensional $U(\g)$-module. We may pay attention that  (unlike the notation in \cite[7.1.4, 7.1.12]{D}) the highest weight of $V(\lambda)$ in our paper is $\lambda$ and not $\lambda-\rho$, where $\rho$ is the sum of all fundamental weights of $(\g,\,\h)$.

\subsection{Semi-direct product}\label{defsemi-direct}

Recall the parabolic subalgebra $\p=\mathfrak r\oplus\m$ and its opposite parabolic subalgebra $\p^-=\mathfrak r\oplus\m^-$, with $\m$, resp. $\m^-$, the nilpotent radical of $\p$, resp. $\p^-$.

We now consider the semi-direct product $\tilde\p=\mathfrak r\ltimes(\m)^a$, resp. $\tilde\p^-=\mathfrak r\ltimes(\m^-)^a$, where $(\m)^a$, resp. $(\m^-)^a$, is isomorphic to $\m$, resp. $\m^-$, as an $\mathfrak r$-module, but where the superscript $a$ means that $(\m)^a$, resp. $(\m^-)^a$, is an abelian ideal of $\tilde\p$, resp. of $\tilde\p^-$.
Such a semi-direct product is still a Lie algebra by \cite[Sect. 4] {Y1} for example, called an In\"on\" u-Wigner contraction, or a one-parameter contraction of $\p$, resp. of $\p^-$. 
The $\Bbbk$-vector space $\tilde\p$, resp. $\tilde\p^-$, is equal to $\p$, resp. $\p^-$, as a vector space and if we denote by $[\,,\,]_{\tilde\p}$, resp. $[\,,\,]_{\tilde\p^-}$ the Lie bracket in $\tilde\p$, resp. $\tilde\p^-$, and  by $[\,,\,]$ the Lie bracket in $\g$, then
one has that
\begin{equation}\forall z,\,z'\in\mathfrak r,\,\forall x,\,x'\in\m,\;\;[z,\,x]_{\tilde\p}=[z,\,x],\;\;[z,\, z']_{\tilde\p}=[z,\,z'],\;\;[x,\,x']_{\tilde\p}=0\label{brackettildep}\end{equation}
\begin{equation}\forall z,\,z'\in\mathfrak r,\,\forall y,\,y'\in\m^-,\;\;[z,\,y]_{\tilde\p^-}=[z,\,y],\;\;[z,\, z']_{\tilde\p^-}=[z,\,z'],\;\;[y,\,y']_{\tilde\p^-}=0.\label{brackettildep-}\end{equation}

\section{The degenerate highest weight module.}

In this section, we fix $\lambda\in P^+(\pi)$ and we will define, from the irreducible highest weight module $V(\lambda)$ of highest weight $\lambda$, some vector space denoted by $\widetilde{V}(\lambda)$ which can be endowed with a left $U(\tilde\p^-)$-module structure, so that it is isomorphic to $V(\lambda)$ as a left $U(\mathfrak r)$-module.

\subsection{The generalized PBW filtration and the degenerate highest weight module $\widetilde{V}(\lambda)$}\label{PBW}

 Consider $V(\lambda)$ the irreducible highest weight $\g$-module of highest weight $\lambda$ as defined in subsection \ref{genotation}.

Generalizing the PBW filtration on a highest weight irreducible $\g$-module introduced in \cite{FFL}, when $\p=\mathfrak b$ is a Borel subalgebra of $\g$ (that is, when $\pi'=\emptyset$), we define what we call {\it the generalized Poincar\'e-Birkhoff-Witt filtration} on $V(\lambda)$ as follows.

Choose $v_{\lambda}$  a nonzero weight vector in $V(\lambda)$ of highest weight $\lambda$ and $v_{w_0\lambda}$ a nonzero weight vector in $V(\lambda)$ of lowest weight $w_0\lambda$.
Since $\n^-=\n^-_{\pi'}\oplus\m^-$, the multiplication in the enveloping algebra gives, by the Poincar\'e-Birkhoff-Witt theorem \cite[2.1.11]{D}, an isomorphism of vector spaces $ U(\n^-_{\pi'})\otimes U(\m^-)\simeq U(\n^-)$. Then we have that $$V(\lambda)=U(\n^-_{\pi'}).(U(\m^-).v_{\lambda})=U(\mathfrak r).(U(\m^-).v_{\lambda})=U(\m^-).(U(\n^-_{\pi'}).v_{\lambda})$$
since $\m^-$ is an ideal of $\p^-$.
Set $V'(\lambda)=U(\n^-_{\pi'}).v_{\lambda}$. The latter is an irreducible $U(\mathfrak r)$-module.

Recall $(U_k(\m^-))_{k\in\mathbb N}$ the canonical filtration (also called standard degree filtration in \cite{FFL})   on the enveloping algebra $U(\m^-)$ of $\m^-$. More precisely $U_k(\m^-)$ is the vector subspace of $U(\m^-)$ generated by 
the products $y_1\cdots y_p$ where $y_i\in\m^-$ for all $i$, $1\le i\le p$, and $p\le k$.

For all $k\in\mathbb N$, let $\mathscr F_k(V(\lambda))$
be the vector subspace of $V(\lambda)$ generated by 
$$\begin{array}{cc}\{v\in V(\lambda)\mid\exists p\in\mathbb N,\,p\le k,\,\exists y_1,\,\ldots,\,y_p\in\m^-,\,\exists  u'\in U(\mathfrak r);\;\\
v= u'\,y_1\cdots y_p.v_{\lambda}\}.\end{array}$$
where $u'\,y_1\cdots y_p$ denotes an element in $U(\p^-)$. 
Observe that we also have that $\mathscr F_k(V(\lambda))$ is the vector subspace of $V(\lambda)$ generated by
$$\begin{array}{cc}\{v\in V(\lambda)\mid\exists p\in\mathbb N,\,p\le k,\,\exists y_1,\,\ldots,\,y_p\in\m^-,\,\exists  u'\in U(\mathfrak r);\;\\
v= y_1\cdots y_p\,u'.v_{\lambda}\}\end{array}$$ since $[\mathfrak r,\,\m^-]\subset\m^-$.

In other words, one has that $\mathscr F_0(V(\lambda))=U(\mathfrak r).v_{\lambda}=U(\n^-_{\pi'}).v_{\lambda}=V'(\lambda)$ and for all $k\in\mathbb N$, $\mathscr F_k(V(\lambda))=U_k(\m^-). V'(\lambda)$
 is a left $U(\mathfrak r)$-module.
Then $\mathscr F:=(\mathscr F_k(V(\lambda)))_{k\in\mathbb N}$ is an increasing and exhaustive filtration on $V(\lambda)$. We call it the generalized Poincar\'e-Birkhoff-Witt filtration on $V(\lambda)$ since when $\pi'=\emptyset$, it coincides with the PBW filtration on $V(\lambda)$ introduced in \cite{FFL}. The associated graded space is denoted by $$\widetilde{V}(\lambda):=gr_{\mathscr F}(V(\lambda))=\bigoplus_{k\in\mathbb N}\frac{\mathscr F_k(V(\lambda))}{\mathscr F_{k-1}(V(\lambda))}$$  where $\mathscr F_{-1}(V(\lambda)):=\{0\}$ and we call $\widetilde{V}(\lambda)$ the {\it degenerate highest weight module associated with $\lambda$}.
For all $v\in\mathscr F_k(V(\lambda))$, we denote by $gr_k(v)$ its canonical image in
$gr_k(V(\lambda)):=\displaystyle\frac{\mathscr F_k(V(\lambda))}{\mathscr F_{k-1}(V(\lambda))}$.
Denote by $\widetilde{V'}(\lambda)$ the canonical image of $V'(\lambda)$ in $\widetilde{V}(\lambda)$ that is, $\widetilde{V'}(\lambda)=gr_0(V'(\lambda))=gr_0(V(\lambda))\subset\widetilde{V}(\lambda)$.

\subsection{Left $U(\tilde\p^-)$-module structure on $\widetilde{V}(\lambda)$}\label{descr} Recall that, for all $k\in\mathbb N$, $\mathscr F_k(V(\lambda))$ is a finite-dimensional left $U(\mathfrak r)$-module and that the Lie algebra $\mathfrak r$ is reductive and  the elements of its centre act reductively in $\mathscr F_k(V(\lambda))$. Then  by \cite[1.6.4]{D} one has that $\mathscr F_k(V(\lambda))$ is a semisimple $U(\mathfrak r)$-module. Moreover $\mathscr F_{k-1}(V(\lambda))$ is a submodule of $\mathscr F_k(V(\lambda))$. Then there exists a left $U(\mathfrak r)$-submodule $\mathscr F^k(V(\lambda))$ of $\mathscr F_k(V(\lambda))$ such that $\mathscr F_k(V(\lambda))=\mathscr F^k(V(\lambda))\oplus\mathscr F_{k-1}(V(\lambda))$ and
we have that $$\mathscr F_k(V(\lambda))=\bigoplus_{i=0}^k\mathscr F^i(V(\lambda))$$
where $\mathscr F^0(V(\lambda))=\mathscr F_0(V(\lambda))$.
One deduces that $$V(\lambda)=\bigoplus_{k\in\mathbb N}\mathscr F^k(V(\lambda)).$$  It allows us to define, for all $k\in\mathbb N$, an isomorphism of vector spaces $$\beta_{\lambda}^k : gr_k(V(\lambda))\longrightarrow \mathscr F^k(V(\lambda))$$ such that, for all $v\in \mathscr F_k(V(\lambda))$, $v=\sum_{i=0}^kv_i$ with $v_i\in\mathscr F^i(V(\lambda))$, for all $0\le i\le k$, $$\beta_{\lambda}^k(gr_k(v))=v_k.$$ Then the direct sum $\beta_{\lambda}=\bigoplus_{k\in\mathbb N}\beta_{\lambda}^k$  is an isomorphism between the vector spaces $\widetilde{V}(\lambda)$ and $V(\lambda)$. 

Set, for all $y\in\m^-$, $z\in\mathfrak r$ and $v\in\mathscr F_k(V(\lambda))$, \begin{equation}y.gr_k(v)=gr_{k+1}(y.v)\label{actionm}\end{equation} and  \begin{equation}z.gr_k(v)=gr_k(z.v).\label{actionr}\end{equation}

We will see below that equations (\ref{actionm}) and (\ref{actionr}) extend to a  left $U(\tilde\p^-)$-module structure on $\widetilde{V}(\lambda)$ 
and that $\beta_{\lambda}$ is an isomorphism of $U(\mathfrak r)$-modules.

Set $\tilde\n^-=\n^-_{\pi'}\ltimes(\m^-)^a$ : it is a Lie subalgebra of $\tilde\p^-$. Set also ${\tilde v}_{\lambda}=gr_0(v_{\lambda})$.

Denote by $\theta:S(\p^-)\longrightarrow U(\p^-)$ the symmetrisation, as defined in \cite[2.4.6]{D}. More precisely for $k\in\mathbb N^*$, and for all $y_1,\ldots,\,y_k\in\p^-$, $$\theta(y_1\cdots y_k)=\frac{1}{k!}\sum_{\sigma\in\mathfrak S_k}y_{\sigma(1)}\cdots y_{\sigma(k)}$$ where $\mathfrak S_k$ is the set of permutations of $k$ elements, the product  in the left hand side lying in $S_k(\p^-)$ and the product   in the right hand side lying in $U_k(\p^-)$. 
 Endow the symmetric algebra $S(\p^-)$, resp. the enveloping algebra $U(\p^-)$, with the adjoint action of $U(\mathfrak r)$, denoted by $ad$, which extends uniquely by derivation the adjoint action of $\mathfrak r$ on $\p^-$ given by Lie bracket. 
  By \cite[2.4.10]{D} the map $\theta$ is an isomorphism of $ad\,U(\mathfrak r)$-modules. For all $k\in\mathbb N$, set $U^k(\m^-)=\theta(S_k(\m^-))$. Then $U^k(\m^-)$ is a left $ad\,U(\mathfrak r)$-submodule of $U_k(\m^-)$ and actually one has that $U_k(\m^-)=U^k(\m^-)\oplus U_{k-1}(\m^-)$ by \cite[2.4.4, 2.4.5]{D}. Denote by $pr_{U^k(\m^-)}$ the projection onto $U^k(\m^-)$ with respect to the above decomposition. 
We have the following. 

\begin{lm}
Let $\lambda\in P^+(\pi)$ and $k\in\mathbb N$.
\begin{enumerate}
\item[(i)] Equations (\ref{actionm}) and (\ref{actionr}) extend to a  left $U(\tilde\p^-)$-action on the vector space  $\widetilde{V}(\lambda)$ and  for this structure we have the following equalities :
\begin{equation}\widetilde{V'}(\lambda)=U(\mathfrak r).\tilde v_{\lambda}\label{V'}\end{equation}
\begin{equation}\widetilde{V}(\lambda)=U(\tilde\p^-).{\tilde v}_{\lambda}=U(\tilde\n^-).{\tilde v}_{\lambda}=S(\m^-).\widetilde{V'}(\lambda)=U(\tilde\p^-).\widetilde{V'}(\lambda).\label{V}\end{equation}
\item[(ii)]
For all $s\in S_k(\m^-)$, $u'\in U(\mathfrak r)$ and $u\in U_k(\m^-)$ one has :
\begin{equation}
 su'.\tilde v_{\lambda}=gr_k(\theta(s)u'.v_{\lambda})\label{first}\end{equation}
 \begin{equation}gr_k(uu'.v_{\lambda})=gr_k(pr_{U^k(\m^-)}(u)u'.v_{\lambda})\label{second}\end{equation} 
 \begin{equation}gr_k(V(\lambda))=S_k(\m^-).\widetilde{V'}(\lambda).\label{third}\end{equation}
 
\item[(iii)] The map $\beta_{\lambda}$ is an isomorphism of $U(\mathfrak r)$-modules between $\widetilde{V}(\lambda)$ and $V(\lambda)$. Then $\widetilde{V'}(\lambda)$ is a left irreducible $U(\mathfrak r)$-module and $\widetilde{V}(\lambda)$ has the same set of weights as $V(\lambda)$, especially $\lambda$ is the highest weight of $\widetilde{V}(\lambda)$  and $w_0\lambda$ is its lowest weight.

\item[(iv)] One may choose $\mathscr F^k(V(\lambda))$ to be included in $U^k(\m^-).V'(\lambda)$.
\end{enumerate}

\end{lm}

\begin{proof}

By \cite[2.1.1]{D} and  (\ref{brackettildep-}) of subsection \ref{defsemi-direct}, one may observe that the algebra $U(\tilde\p^-)$ is the quotient of the tensor algebra $T(\tilde\p^-)= T(\p^-)$ of the vector space $\tilde\p^-=\p^-$ by the two-sided ideal generated by the set
$$\{z\otimes z'-z'\otimes z-[z,\,z'],\;z\otimes y-y\otimes z-[z,\,y],\;y\otimes y'-y'\otimes y;\;z,\,z'\in\mathfrak r,\, y,\,y'\in\m^-\}$$
and that, by the Poincar\'e-Birkhoff-Witt theorem \cite[2.1.11]{D}, the multiplication is an isomorphism between the $\Bbbk$-vector spaces $U(\mathfrak r)\otimes S(\m^-)$  and $U(\tilde\p^-)$.

Fix $k\in\mathbb N$.
 For all $x\in\m^-\oplus\mathfrak r=\p^-$, denote by $x.\mathscr F_k(V(\lambda))$ the vector subspace of $V(\lambda)$ formed by all the vectors $x.v$, with $v\in\mathscr F_k(V(\lambda))$ (where $x.v$ denotes the action of $x$ on $v$ by the left $U(\g)$-module structure on $V(\lambda)$).
 
Then for all $y\in\m^-$, one has that $y.\mathscr F_k(V(\lambda))\subset\mathscr F_{k+1}(V(\lambda))$, and for all $z\in\mathfrak r$, one has that $z.\mathscr F_k(V(\lambda))\subset\mathscr F_{k}(V(\lambda))$. It follows that equation (\ref{actionm}) extends to a left action of $S(\m^-)$ on $\widetilde{V}(\lambda)$ 
 since moreover, for $y,\,y'\in\m^-$ and $v\in\mathscr F_k(V(\lambda))$, we have: $$y.(y'.(gr_k(v))-y'.(y.gr_k(v))=gr_{k+2}((yy'-y'y).v)=gr_{k+2}([y,\,y'].v)=0.$$
Similarly equation (\ref{actionr})
extends to a left action of $U(\mathfrak r)$ on $\widetilde{V}(\lambda)$ induced by the left action of $U(\mathfrak r)$ on $V(\lambda)$. 
Finally both equations (\ref{actionm}) and (\ref{actionr}) extend to  a left action of $U(\tilde\p^-)$ on $\widetilde{V}(\lambda)$ (by say, \cite[2.2.1, 2.2.2]{D}). Equation (\ref{V'}) follows since $\widetilde{V'}(\lambda)=gr_0(V'(\lambda))=gr_0(U(\mathfrak r).v_{\lambda})$.

 Let $\tilde v\in\widetilde{V}(\lambda)$. There exists $k\in\mathbb N$ and $v_i\in\mathscr F_i(V(\lambda))$, for $0\le i\le k$, such that $\tilde v=\sum_{i=0}^k gr_i(v_i)$ with, for all $i$, $v_i=\sum_{j=1}^{n_i}u'_{ij}u_{ij}.v_{\lambda}$ where $u'_{ij}\in U(\mathfrak r)$ and $u_{ij}\in U_i(\m^-)$. Then by equation (\ref{actionr}), one has $$gr_i(v_i)=\sum_{j=1}^{n_i}u'_{ij}.gr_i(u_{ij}.v_{\lambda})$$ and by equation (\ref{actionm}), $$gr_i(u_{ij}.v_{\lambda})\in S_i(\m^-).gr_0(v_{\lambda}).$$ 
 Actually we may take the $u'_{ij}$ in $U(\n^-_{\pi'})$, since $$V(\lambda)=U(\n^-).v_{\lambda}=U(\n^-_{\pi'}).(U(\m^-).v_{\lambda}).$$
 We then have
 $\widetilde{V}(\lambda)=U(\tilde\p^-).\tilde v_{\lambda}=U(\tilde\n^-).\tilde v_{\lambda}$. Since $\widetilde{V'}(\lambda)=U(\mathfrak r).\tilde v_{\lambda}$ and since the multiplication gives the isomorphism of  vector spaces $U(\tilde\p^-)\simeq S(\m^-)\otimes U(\mathfrak r)$, we also have that $\widetilde{V}(\lambda)=S(\m^-).\widetilde{V'}(\lambda)=U(\tilde\p^-).\widetilde{V'}(\lambda)$. Hence equation (\ref{V}).
 
 Let $k\in\mathbb N^*$ and set $s=y_1\cdots y_k\in S_k(\m^-)$ with $y_i\in\m^-$ for all $1\le i\le k$. Then $\theta(s)=y_1\cdots y_k+u\in U^k(\m^-)$ with $u\in U_{k-1}(\m^-)$ and $y_1\cdots y_k\in U_k(\m^-)$. 
 Then equations (\ref{actionm}) and (\ref{actionr}) and the fact that $U_{k-1}(\m^-)U(\mathfrak r).v_{\lambda}=\mathscr F_{k-1}(V(\lambda))$ give equation (\ref{first}). Equation (\ref{second}) is obvious by the decomposition $U_k(\m^-)=U^k(\m^-)\oplus U_{k-1}(\m^-)$. Both equations imply equation (\ref{third}). 
  
 By equation (\ref{actionr}), we have that $gr_k(V(\lambda))$ is an $U(\mathfrak r)$-module. Moreover if $\tilde v\in gr_k(V(\lambda))$ is such that $\tilde v=gr_k(v_k)$ with $v_k\in\mathscr F^k(V(\lambda))$, we have that $\beta_{\lambda}^k(\tilde v)=v_k$. Let $z\in\mathfrak r$. Then by equation (\ref{actionr}), one has that $z.\tilde v=gr_k(z.v_k)$ which implies that \begin{equation*}\beta_{\lambda}^k(z.\tilde v)=z.v_k=z.\beta_{\lambda}^k(\tilde v),\end{equation*} since $z.v_k\in\mathscr F^k(V(\lambda))$ because $\mathscr F^k(V(\lambda))$ is an $U(\mathfrak r)$-module. This shows $(iii)$.
 Finally to prove $(iv)$ it suffices to observe that $U^k(\m^-).V'(\lambda)$ is a finite dimensional $U(\mathfrak r)$-module. Set $W_k=U^k(\m^-).V'(\lambda)\cap\mathscr F_{k-1}(V(\lambda))$. Then $W_k$ is a left $U(\mathfrak r)$-submodule of $U^k(\m^-).V'(\lambda)$ and then there exists a left $U(\mathfrak r)$-submodule $W'_k$ such that $U^k(\m^-).V'(\lambda)=W_k\oplus W'_k$.
 Now $W'_k\cap\mathscr F_{k-1}(V(\lambda))=\{0\}$ and then one may choose the $U(\mathfrak r)$-module $\mathscr F^k(V(\lambda))$ to contain $W'_k$. But $\mathscr F_k(V(\lambda))=\mathscr F_{k-1}(V(\lambda))\oplus\mathscr F^k(V(\lambda))=U_k(\m^-).V'(\lambda)\subset U^k(\m^-).V'(\lambda)+U_{k-1}(\m^-).V'(\lambda)$ since $U_k(\m^-)=U ^k(\m^-)\oplus U_{k-1}(\m^-)$. It follows that $W'_k=\mathscr F^k(V(\lambda))$, which completes the proof.
 \end{proof}
 
 \subsection{Left $U(\p)$-module structure on $\widetilde{V}(\lambda)$.}\label{pmod}
 
 Recall the isomorphism $\beta_{\lambda}$ of $U(\mathfrak r)$-modules from $\widetilde{V}(\lambda)$ into $V(\lambda)$ (lemma \ref{descr} $(iii)$). For all $\tilde v\in\widetilde{V}(\lambda)$ and all $x\in\p$, one sets \begin{equation}\rho_{\lambda}(x)(\tilde v)=\beta_{\lambda}^{-1}(x.\beta_{\lambda}(\tilde v))\label{actionp}\end{equation}
 where $x.\beta_{\lambda}(\tilde v)$ stands for the left action of $x\in\g$ on $\beta_{\lambda}(\tilde v)\in V(\lambda)$.
 
 It is easily checked that $\rho_{\lambda}$ is a morphism of Lie algebras from $\p$ to $\mathfrak{gl}(\widetilde{V}(\lambda))$, hence that it extends to a left action of $U(\p)$ on $\widetilde{V}(\lambda)$ (again by \cite[2.2.1, 2.2.2]{D}). Moreover for all $x\in\mathfrak r$ and $\tilde v\in\widetilde{V}(\lambda)$, since $\beta_{\lambda}$ is a morphism of $U(\mathfrak r)$-modules, one has that $\rho_{\lambda}(x)(\tilde v)=x.\tilde v$ where the right hand side denotes the left action of $\mathfrak r$ on $\widetilde{V}(\lambda)$ defined in subsection \ref{descr}.
 
 \begin{Rq}\rm
 By \cite[2.7]{FJ2}, one has that $V'(\lambda)=\{v\in V(\lambda)\mid \m.v=0\}$. Hence \begin{equation}\widetilde{V'}(\lambda)=\{\tilde v\in\widetilde{V}(\lambda)\mid \rho_{\lambda}(\m)(\tilde v)=0\}\label{anulm}\end{equation}
 since $\widetilde{V'}(\lambda)=\beta_{\lambda}^{-1}(V'(\lambda))$.

 \end{Rq}
 
 \section{Action of a smash product  on $U(\tilde\p^-)$}\label{structuresup}
 
 In this section, we will define a smash product $A=T(\m)\string # U(\mathfrak r)$, containing the enveloping algebra $U(\mathfrak r)$ and the tensor algebra $T(\m)$, where the action of $U(\mathfrak r)$ on $T(\m)$ derives from the adjoint action of $\mathfrak r$ in $T(\m)$ which extends uniquely by derivation the adjoint action given by Lie bracket. This algebra $A$ is an associative algebra, which is actually a Hopf algebra.
 We will define what we call a generalized adjoint action (denoted by $ad^{**}$) of the algebra $A$ on the enveloping algebra $U(\tilde\p^-)$ and another left action of $A$ on $U(\tilde\p^-)$, where the latter is simply left multiplication when restricted to $U(\mathfrak r)$. The action $ad^{**}$ derives from the coadjoint action, denoted by $ad^*$, of $\tilde\p$ on $\p^-$ (note that, as vector spaces, one has ${\tilde\p}^*\simeq\p^-$). We will see in subsection \ref{dualS} why we need to take this coadjoint action $ad^*$.
  
  \subsection{A smash product}\label{smash}
 Recall that $\m$ denotes the nilpotent radical of $\p$ and that $T(\m)$ denotes the tensor algebra of $\m$. Since $[\mathfrak r,\,\m]\subset\m$, the algebra $T(\m)$ is an $U(\mathfrak r)$-algebra (in the sense of \cite[1.1.6]{J2}) with the adjoint action of $\mathfrak r$ on $T(\m)$ (denoted by $ad$) extending by derivation the adjoint action of $\mathfrak r$ on $\m$ given by the Lie bracket in $\g$. Then 
 we may consider the Hopf smash product $A=T(\m)\string # U(\mathfrak r)$ in the sense of \cite[1.1.8]{J2}. More precisely $A$  is equal as a vector space to the tensor product $T(\m)\otimes U(\mathfrak r)$, with multiplication given by $(s\otimes u)(s'\otimes u')=s\,ad\,u_1(s')\otimes u_2u'$ where $\Delta(u)=u_1\otimes u_2$ (Sweedler notation), $\Delta$ being the coproduct in $U(\mathfrak r)$, $s,\,s'\in T(\m)$ and $u,\, u'\in U(\mathfrak r)$.
 
 For example for all $z\in\mathfrak r$, $s,\,s'\in T(\m)$ and $u\in U(\mathfrak r)$, one has that $(s'\otimes z)(s\otimes u)=s'\,ad\,z(s)\otimes u+s's\otimes zu$. By setting $s\otimes 1=s$ and $1\otimes u=u$ we may view $T(\m)$ and $U(\mathfrak r)$ as subalgebras of $A$.  Then one has in $A$ that $s\otimes u=(s\otimes 1)(1\otimes u)=su$ and that 
 \begin{equation}\forall z\in\mathfrak r,\;\forall s\in T(\m),\;ad\,z(s)=zs-sz \label{equationA}\end{equation}
 and in particular
 \begin{equation}\forall z\in\mathfrak r,\;\forall x\in\m,\;[z,\,x]=zx-xz.\label{equationmA}\end{equation}
 
 Observe that $A$ is an associative  unitary  algebra (see \cite[1.1.8]{J2}) which is also a bialgebra thanks to the coproducts in $T(\m)$ and in $U(\mathfrak r)$. More precisely denoting also by $\Delta$ the coproduct in $T(\m)$, and by $\Delta_A$ the coproduct in $A$, we set for $s\in T(\m)$ and $u\in U(\mathfrak r)$, $\Delta_A(s\otimes u)=(s_1\otimes u_1)\otimes (s_2\otimes u_2)$ if $\Delta(s)=s_1\otimes s_2$ and $\Delta(u)=u_1\otimes u_2$ with Sweedler notation. We then have that $\Delta_A((s\otimes 1)(1\otimes u))=\Delta_A(s\otimes 1)\Delta_A(1\otimes u)$ and more generally for $s,\,s'\in T(\m)$ and $u,\,u'\in U(\mathfrak r)$, $\Delta_A((s\otimes u)(s'\otimes u'))=\Delta_A(s\otimes u)\Delta_A(s'\otimes u')$ by the cocommutativity of $\Delta$. Note that the coproduct $\Delta_A$ extends the coproduct $\Delta$ in $T(\m)$ and in $U(\mathfrak r)$.
 Actually the bialgebra $A$ is a Hopf algebra with the coidentity $\varepsilon$ given by $\varepsilon(x)=0$ for all $x\in\p$ and the antipode given by $a\in A\mapsto a^{\top}\in A$, where
 \begin{equation}a^{\top}=(-1)^rx_r\cdots x_1\in A\label{antipode}\end{equation}
 if $a=x_1\cdots x_r\in A$ (product in $A$) with $x_1,\,\ldots,\,x_r\in\p$  extended  by linearity to every element in $A$. One checks easily that the coidentity and the antipode (which coincide respectively with the coidentity and the antipode on $T(\m)$ and on $U(\mathfrak r)$, see for instance \cite[1.2.5]{J2}) are compatible with equation (\ref{equationmA}) which defines the smash product $A$.

  Roughly speaking, the Hopf algebra $A$ coincides with the enveloping algebra $U(\p)$ or even $U(\tilde\p)$, except that no relations are required for the associative product of elements in $\m$.

 \subsection{The coadjoint action of $\tilde\p$ on $\p^-$}\label{coadjointaction}
 Recall the opposite parabolic subalgebra $\p^-$ of $\p$. Thanks to the Killing form on $\g$, we have the isomorphism of vector spaces $\tilde\p^*\simeq \p^-$. As it was already mentioned in \cite[2]{P}, $\p^-$ is a $\tilde\p$-module, by the socalled coadjoint representation (denoted by $ad^*$) of $\tilde\p=\mathfrak r\ltimes(\m)^a$ in $\p^-$ defined as follows.
 
 \begin{equation}\forall x\in\mathfrak r,\;\forall \;y\in\p^-,\;ad^*x(y)=[x,\,y].\label{eg2}\end{equation}
 
\begin{equation}\forall x\in\m,\;\forall\;y\in\p^-,\;ad^*x(y)=pr_{\mathfrak r}([x,\,y])\label{eg}\end{equation} where $pr_{\mathfrak r}$ is the projection of $\g=\mathfrak r\oplus\m\oplus\m^-$ onto $\mathfrak r$.
 In particular \begin{equation}\forall x\in\m,\;\forall\;y\in\mathfrak r,\;ad^*x(y)=0.\label{eg1}\end{equation}

\begin{lm}
The map $ad^*:\tilde\p\longrightarrow\g\mathfrak l(\p^-)$ is a morphism between the Lie algebras $\tilde\p$ and $\g\mathfrak l(\p^-)$. In other words it gives a representation of $\tilde\p$ in $\p^-$, which extends uniquely  to a representation of $U(\tilde\p)$ in $\p^-$. This representation also extends uniquely by derivation to a representation of $U(\tilde\p)$ in the symmetric algebra $S(\p^-)$, which we still denote by $ad^*$.

\end{lm}

\begin{proof}
We give a proof of the lemma  for the reader's convenience.
It suffices  to prove that, for all $x,\,x'\in\tilde\p$, and for all $y\in\p^-$, we have \begin{equation}(ad^*x\circ ad^*x')(y)-(ad^*x'\circ ad^*x)(y)-ad^*[x,\,x']_{\tilde\p}(y)=0.\tag{$\star$}\label{eg3}\end{equation}

Assume that $x,\,x'\in\m$. Then $[x,\,x']_{\tilde\p}=0$ by equation (\ref{brackettildep}) in subsection \ref{defsemi-direct}. Moreover for all $y\in\p^-$, \begin{equation*}(ad^*x\circ ad^*x')(y)=ad^*x(pr_{\mathfrak r}([x',\,y]))=0\end{equation*} by (\ref{eg}) and (\ref{eg1}). Then equality (\ref{eg3}) follows in this case.

Assume that $x,\,x'\in\mathfrak r$. Then equality (\ref{eg3}) follows from equations (\ref{brackettildep}) and (\ref{eg2}).

It remains to prove equality (\ref{eg3}) for $x\in\m$ and $x'\in\mathfrak r$. By equations (\ref{brackettildep}), (\ref{eg2}) and (\ref{eg}) one has that, for all $y\in\p^-$,
\begin{align*}&(ad^*x\circ ad^*x')(y)-(ad^*x'\circ ad^*x)(y)-ad^*[x,\,x']_{\tilde\p}(y)\\
&=ad^*x([x',\,y])-ad^*x'(pr_{\mathfrak r}([x,\,y]))-pr_{\mathfrak r}([[x,\,x'],\,y])\\
&=pr_{\mathfrak r}([x,\,[x',\,y]])-[x',\,pr_{\mathfrak r}([x,\,y])]-pr_{\mathfrak r}([[x,\,x'],\,y]).
\end{align*}

Denote by $pr_{\m}$, resp. $pr_{\m^-}$, the projection of $\g=\mathfrak r\oplus\m\oplus\m^-$ onto $\m$, resp. onto $\m^-$.

 Then \begin{equation}[x',\,pr_{\mathfrak r}([x,\,y])]
 =\bigl[x',\,[x,\,y]-pr_{\m^-}([x,\,y])-pr_{\m}([x,\,y])\bigr]\tag{$\star\star$}\label{eg4}\end{equation}
 
 and we have
 \begin{equation}[x',\,pr_{\m^-}([x,\,y])]\in\m^-,\;[x',\,pr_{\m}([x,\,y])]\in\m,\;[x',\,pr_{\mathfrak r}([x,\,y])]\in\mathfrak r.\tag{$\star\star\star$}\label{eg5}\end{equation}
 
 Then by (\ref{eg4}) and (\ref{eg5}) we have that
 \begin{equation}[x',\,pr_{\mathfrak r}([x,\,y])]=pr_{\mathfrak r}([x',\,[x,\,y]]).\tag{$\star\star\star\,\star$}\label{eg6}\end{equation}
 
 It follows by (\ref{eg6}) that
\begin{align*} (ad^*x\circ ad^*x')(y)-(ad^*x'\circ ad^*x)(y)-ad^*[x,\,x']_{\tilde\p}(y)\\
=pr_{\mathfrak r}\bigl([x,\,[x',\,y]]-[x',\,[x,\,y]]-[[x,\,x'],\,y]\bigr)=0\end{align*}
by Jacobi identity in $\g$. \par\noindent
 Applying \cite[2.2.1]{D} and \cite[1.2.14]{D} completes the proof of the lemma.
 \end{proof}

 \subsection{Partial symmetrisation}\label{partialsym}

 Recall the symmetrisation $\theta:S(\p^-)\longrightarrow U(\p^-)$ which is an isomorphism of $ad\,U(\mathfrak r)$-modules, when $S(\p^-)$ and $U(\p^-)$ are endowed with the adjoint action $ad$ (see \ref{descr}). Denote also by $ad$ the adjoint action of $U(\mathfrak r)$ on $U(\tilde\p^-)$, extending by derivation  the Lie bracket of $\mathfrak r$ on $\tilde\p^-$ (see equation (\ref{brackettildep-})).
 
 Set $$\tilde\theta=Id_{S(\m^-)}\otimes\theta_{\mid S(\mathfrak r)}:S(\p^-)\simeq S(\m^-)\otimes S(\mathfrak r)\longrightarrow U(\tilde\p^-)\simeq S(\m^-)\otimes U(\mathfrak r),$$
 
 that is, \begin{equation}\forall s\in S(\m^-),\,\,\forall s'\in S(\mathfrak r),\;\;\tilde\theta(ss')=s\,\theta(s').\label{th}\end{equation}

We call the map $\tilde\theta$ a {\it partial symmetrisation}. Observe that $\tilde\theta$ does not coincide with the symmetrisation $\tilde{\tilde\theta}$ of $S(\tilde\p^-)=S(\p^-)$ in $U(\tilde\p^-)$. For instance, for $y\in\m^-$ and $z\in\mathfrak r$, one has that $\tilde\theta(yz)=yz$, while $\tilde{\tilde\theta}(yz)=\frac{1}{2}(yz+zy)=yz+\frac{1}{2}[z,\,y]$.

 \begin{lm}
 
 The map $\tilde\theta$ is an isomorphism of $ad\,U(\mathfrak r)$-modules, when $S(\p^-)$ and $U(\tilde\p^-)$  are endowed with the adjoint action.
 
 \end{lm}
 
 \begin{proof}
 Since $Id_{S(\m^-)}$ and $\theta_{\mid S(\mathfrak r)}:S(\mathfrak r)\longrightarrow U(\mathfrak r)$ are isomorphisms,
 it follows that $\tilde\theta$ is an isomorphism too.
 
 Let $z\in\mathfrak r$, $s\in S(\m^-)$ and $s'\in S(\mathfrak r)$. Observe that one has that $ad\,z(s)\in S(\m^-)$, and this element may be viewed equally as an element in $S(\p^-)$ or in $U(\tilde\p^-)$. Moreover $ad\,z(s')\in S(\mathfrak r)$ and $ad\,z(\theta(s'))\in U(\mathfrak r)$.
 Since $ad\,z$ is a derivation and since $\theta$ is a morphism of $ad\,U(\mathfrak r)$-modules we have, by equation (\ref{th}):
 \begin{align*}
\tilde\theta (ad\,z(ss'))&={\tilde\theta}(ad\,z(s)s'+s\,ad\,z(s'))\\
  &=ad\,z(s)\theta(s')+s\,\theta(ad\,z(s'))\\
&=ad\,z(s)\theta(s')+s\,ad\,z(\theta(s'))\\
 &=ad\,z\,(s\,\theta(s'))\\
 &=ad\,z(\tilde\theta(ss'))
 \end{align*}
This proves the lemma. \end{proof}

\subsection{Generalized adjoint action of $A$ on $U(\tilde\p^-)$}\label{adjaction}

Recall the isomorphism of vector spaces $U(\tilde\p^-)\simeq S(\m^-)\otimes U(\mathfrak r)$. Then one has that
\begin{equation}U(\tilde\p^-)\simeq\bigoplus_{k\in\mathbb N}S_k(\m^-)\otimes U(\mathfrak r)\label{directsum}\end{equation}
and for all $k\in\mathbb N$, \begin{equation}U_k(\tilde\p^-)\simeq \bigoplus_{0\le j\le k}S_j(\m^-)\otimes U_{k-j}(\mathfrak r)\label{directsumbis}\end{equation}
as vector spaces.
Recall also the coadjoint representation of $\tilde\p$ in the symmetric algebra $S(\p^-)$, which we have denoted by $ad^*$ (subsection \ref{coadjointaction}). Fix $k$ and $j$ in $\mathbb N$ and set $S_{-1}(\m^-)=\{0\}$.
One has that 
\begin{equation}\forall x\in\m,\;\forall s\in S_k(\m^-),\;ad^*x(s)\in S_{k-1}(\m^-)\mathfrak r\subset S_k(\p^-)\label{ad*}\end{equation} by equation (\ref{eg}).
Then one has that 
\begin{multline}\forall x\in\m,\;\forall s\in S_k(\m^-),\;\forall u'\in U_j(\mathfrak r),\\
\tilde\theta(ad^*x(s))u'\in S_{k-1}(\m^-)U_{j+1}(\mathfrak r)\subset U_{k+j}(\tilde\p^-)\label{tildethetaad*}\end{multline} and  that 
\begin{equation}\forall z\in\mathfrak r,\;\forall s\in S_k(\m^-),\;\forall u'\in U(\mathfrak r),\\
ad\,z(su')\in S_k(\m^-)U(\mathfrak r)\subset U(\tilde\p^-).\end{equation}

We set

\begin{equation}\forall x\in\m,\,\forall s\in S_k(\m^-),\,\forall u'\in U(\mathfrak r),\;\;ad^{**}x(su')=\tilde\theta(ad^*x(s))u'\in U(\tilde\p^-)\label{actionA}\end{equation}

and 
\begin{equation}\forall z\in\mathfrak r,\,\forall s\in S_k(\m^-),\,\forall u'\in U(\mathfrak r),\;\;ad^{**}z(su')=ad\,z(su')\in U(\tilde\p^-).\label{actionAr}\end{equation}

Observe that \begin{equation}\forall x\in\m,\,\forall u'\in U(\mathfrak r),\,ad^{**}x(u')=0\label{actionnul}\end{equation} and \begin{equation}\forall x\in\m,\,\forall s\in S_k(\m^-),\,\forall u'\in U(\mathfrak r),\;ad^{**}x(su')=ad^{**}x(s)u'.\label{actionpart}\end{equation}
\begin{lm}

 Equations (\ref{actionA}) and (\ref{actionAr})  extend to a left action of $A$ on the enveloping algebra $U(\tilde\p^-)$. We call this action {\it the generalized adjoint action} of $A$ on $U(\tilde\p^-)$.

\end{lm}
\begin{proof}

Since equation  (\ref{actionAr}) is just the adjoint action, it extends to a left action of $U(\mathfrak r)$ on $U(\tilde\p^-)$ by \cite[2.2.1,\;2.4.9]{D}.

Now consider $x\in\m$. One can extend equation (\ref{actionA})  by linearity so that $ad^{**}x\in End(U(\tilde\p^-))$.  Let us explain why  this is well defined. 

Note first that, for $y\in\m^-$ and $z\in\mathfrak r$, one sets \begin{equation*}ad^{**}x(zy)=ad^{**}x(yz)+ad^{**}x([z,\,y]).\end{equation*}

More generally by equation (\ref{directsum}) every element in $U(\tilde\p^-)$ may be written in the form $\sum_{i\in I} s_i u'_i$ with $I$ a finite set and for all $i\in I$, $s_i\in S(\m^-)$ and $u'_i\in U(\mathfrak r)$, with the $s_i$, $i\in I$, linearly independent. Then if such an element is zero, we have that $u'_i=0$ for all $i\in I$,
 and then $ad^{**}x(\sum_{i\in I}s_iu'_i)=\sum_{i\in I}\tilde\theta(ad^*x(s_i))u'_i=0$. 
 
Moreover for $s,\,s'\in S(\m^-)$, one has \begin{equation*}ad^{**}x(ss')=\tilde\theta(ad^*x(ss'))=\tilde\theta(ad^*x(s's))=ad^{**}x(s's)\end{equation*}
since $ad^{*}x$ is an endomorphism of $S(\p^-)$. Then $ad^{**}x$ is well defined on $U(\tilde\p^-)$.

Finally $ad^{**}$ is a linear map from $\m$ to $End(U(\tilde\p^-))$, which  extends naturally to a $\Bbbk$-algebras morphism from $T(\m)$ to $End(U(\tilde\p^-))$. We then obtain  left $U(\mathfrak r)$ and $T(\m)$-module structures on $U(\tilde\p^-)$.

It remains to check that both structures imply a left $A$-module structure on $U(\tilde\p^-)$, that is, are compatible with equation (\ref{equationA}), which defines the smash product $A$. For this it suffices to prove that 
\begin{equation}\forall x\in\m,\;\forall z\in\mathfrak r,\;ad^{**}z\circ ad^{**}x-ad^{**}x\circ ad^{**}z=ad^{**}[z,\,x].\label{comp}\end{equation}

Let $x\in\m$, $z\in\mathfrak r$, $s\in S(\m^-)$ and $u'\in U(\mathfrak r)$.
Since $ad\,z$ is a derivation in $U(\tilde\p^-)$, one has that

\begin{align*}
ad^{**}(zx-xz)(su')&=(ad^{**}z\circ ad^{**}x-ad^{**}x\circ ad^{**}z)(su')\\
&=ad\,z(\tilde\theta(ad^*x(s))u')\\
&-ad^{**}x(ad\,z(s)u'+s\,ad\,z(u'))\\
&=ad\,z(\tilde\theta(ad^*x(s)))u'+\tilde\theta(ad^*x(s))ad\,z(u')\\
&-\tilde\theta(ad^*x(ad\,z(s)))u'
-\tilde\theta(ad^*x(s))ad\,z(u')
\end{align*}
since moreover $ad\,z(s)\in S(\m^-)$ and $ad\,z(u')\in U(\mathfrak r)$.
Then 
\begin{equation*}
ad^{**}(zx-xz)(su')=\tilde\theta((ad\,z\circ ad^*x)(s))u'-\tilde\theta(ad^*x(ad\,z(s)))u'
\end{equation*} by lemma \ref{partialsym}
and then 
\begin{equation*}
ad^{**}(zx-xz)(su')=\tilde\theta((ad^*z\circ ad^*x)(s))u'-\tilde\theta((ad^*x\circ ad^*z)(s))u'
\end{equation*}
since, for all $t\in S(\p^-)$, one has that $ad\,z(t)=ad^*z(t)$ by equation (\ref{eg2}).

Since \begin{equation*}ad^*z\circ ad^*x-ad^*x\circ ad^*z=ad^*[z,\,x]\end{equation*} in $End(S(\p^-))$ by the proof of Lemma \ref{coadjointaction}, the required equation (\ref{comp}) follows.
\end{proof}

\begin{Rq}
We will see in subsection \ref{lractionsA} why we call this left action $ad^{**}$ of $A$ on $U(\tilde\p^-)$ a generalized adjoint action.

\end{Rq}
\subsection{Left and right actions of $A$ on $U(\tilde\p^-)$}\label{lractionsA}
Here we will define a right, resp. a left action, of $A$ on $U(\tilde\p^-)$ as follows.
\begin{equation}\forall u'\in U(\mathfrak r),\;\forall u\in U(\tilde\p^-),\;R(u')(u)=uu'\label{ractionr}\end{equation} 
(product in $U(\tilde\p^-)$).
Then $R_{\mid U(\mathfrak r)}$ is a right action of $U(\mathfrak r)$ on $U(\tilde\p^-)$ called the regular right action (see \cite[2.2.21]{D}).
We  extend this right action  by setting 
\begin{equation}\forall x\in\m,\;R(x)=0.\label{ractionm}\end{equation}
 One checks immediately that the map $R$ induces a right action of $A$ on $U(\tilde\p^-)$ (still denoted by $R$). It follows that the map $a\in A\mapsto R(a^{\top})$ is a left action of $A$ on $U(\tilde\p^-)$.

One also sets:
\begin{equation}\forall u'\in U(\mathfrak r),\;\forall u\in U(\tilde\p^-),\;L(u')(u)=u'u\label{lactionr}\end{equation} 
(product in $U(\tilde\p^-)$).

Then $L_{\mid U(\mathfrak r)}$ is a left action of $U(\mathfrak r)$ on $U(\tilde\p^-)$ called the regular left action (see \cite[2.2.21]{D}). We extend this left action  by setting
\begin{equation}\forall x\in\m,\;L(x)=ad^{**}x\label{lactionm}\end{equation}  (see equation (\ref{actionA})), which extends by the proof of lemma \ref{adjaction} to a left action of $T(\m)$ on $U(\tilde\p^-)$.

\begin{lm}
The map $L$ extends to a left action of $A$ on $U(\tilde\p^-)$ (still denoted by $L$). Note that this is {\it not} in general a left action of $U(\p)$ nor of $U(\tilde\p)$ on $U(\tilde\p^-)$.
\end{lm}
\begin{proof}
We have to check that the map $L$ preserves equation (\ref{equationA}) which defines the smash product $A$ and for this it suffices to check that $L$ preserves equation (\ref{equationmA}). In other words we have to check that \begin{equation}\forall z\in\mathfrak r,\;\forall x\in\m,\;L([z,\,x])=L(z)\circ L(x)-L(x)\circ L(z).\label{comp2}\end{equation} 
Let $x\in\m$, $z\in\mathfrak r$, $s\in S_k(\m^-)$ and $u'\in U(\mathfrak r)$.
One has that
\begin{align*}
(L(z)\circ L(x)-L(x)\circ L(z)-L([z,\,x]))(su')\\
=z\,ad^{**}x(su')-ad^{**}x(zsu')-ad^{**}[z,\,x](su')\\
=z\,ad^{**}x(su')-(ad^{**}x\circ ad^{**}z)(su')-ad^{**}x(su'z)\\
-ad^{**}[z,\,x](su')
\end{align*}

 since $ad^{**}z(su')=ad\,z(su')=zsu'-su'z$ in $U(\tilde\p^-)$.
 
 Recall that $ad^{**}x(su'z)=\tilde\theta(ad^*x(s))u'z=ad^{**}x(su')z$. Then
 \begin{align*}
(L(z)\circ L(x)-L(x)\circ L(z)-L([z,\,x]))(su')\\
=z\,ad^{**}x(su')-(ad^{**}x\circ ad^{**}z)(su')-ad^{**}x(su')z\\
-ad^{**}[z,\,x](su')\\
=(ad^{**}z\circ ad^{**}x)(su')-(ad^{**}x\circ ad^{**}z)(su')-ad^{**}[z,\,x](su')
\end{align*}
since in $U(\tilde\p^-)$ we have $(ad^{**}z\circ ad^{**}x)(su')=z\,ad^{**}x(su')-ad^{**}x(su')\,z$.
Now equation (\ref{comp}) gives equation (\ref{comp2}).
\end{proof}

Recall (see \cite[1.3.1]{J2} for instance) that adjoint action in a Hopf algebra $A$ may be expanded by using  the right action $R$ and the left action $L$ of $A$  as in the following proposition.
Hence we may view $ad^{**}$ as a generalized adjoint action (here the Hopf algebra $A$ does not act on itself but on $U(\tilde\p^-)$).

  \begin{prop}
  One has the following.
   \begin{equation}\forall a\in A,\;\;ad^{**}a=L(a_1)\circ R(a_2^{\top})\label{adjointaction}\end{equation}
   where $\Delta_A(a)=a_1\otimes a_2$ (with Sweedler notation).
  
  \end{prop}
\begin{proof}
Observe that, as vector spaces, one has \begin{equation*}A\simeq U(\mathfrak r)\oplus\m\otimes T(\m)\otimes U(\mathfrak r).\end{equation*}
Let $a\in U(\mathfrak r)$. Then in this case, $ad^{**}a=ad\,a$  and for all $u\in U(\tilde\p^-)$ one has $(L(a_1)\circ R(a_2^{\top}))(u)=a_1ua_2^{\top}$ (product in $U(\tilde\p^-)$).
Hence the required equation (\ref{adjointaction}) in this case, by \cite[1.3.1]{J2}.

Assume now that $a=uu'$ with $u\in\m\otimes T(\m)$ and $u'\in U(\mathfrak r)$. Set $\Delta_A(u)=u_1\otimes u_2$ and $\Delta_A(u')=u'_1\otimes u'_2$. We have that $u'_1,\,u'_2\in U(\mathfrak r)$ and $u_1,\,u_2\in T(\m)$ and $\Delta_A(a)=u_1u'_1\otimes u_2u'_2$. Moreover one has that $\Delta_A(u)=u\otimes 1+\sum_{i\in I} u_{1i}\otimes u_{2i}$ with for all $i\in I$,  $u_{2i}\in\m\otimes T(\m)$.
But if $u_{2i}\in\m\otimes T(\m)$ then $R({u}^{\top}_{2i})=0$. It follows that, for all $v\in U(\tilde\p^-)$, one has  
\begin{align*}(L(a_1)\circ R(a_2^{\top}))(v)&=(L(uu'_1)\circ R({u'}_2^{\top}))(v)\\
&=(L(u)\circ L(u'_1)\circ R({u'}_2^{\top}))(v)\\
&=(L(u)\circ ad^{**} u')(v)\\
&=(ad^{**}u\circ ad^{**} u')(v)\\
&=ad^{**} a(v)
\end{align*} 
which completes the proof.
\end{proof}

\begin{Rq}

Actually one also has that 
\begin{equation}\forall a\in A,\,\forall b\in A,\,R(b)\circ L(a)=L(a)\circ R(b).\label{commRL}\end{equation}
\rm Indeed 
if $a,\,b\in U(\mathfrak r)$ or if $b\in \m\otimes T(\m)\otimes U(\mathfrak r)$ then equation (\ref{commRL}) is immediate by the associativity of the product in $U(\tilde\p^-)$ or because, in the second case, that $R(b)=0$. Finally when $b\in U(\mathfrak r)$ and $a\in\m\otimes T(\m)\otimes U(\mathfrak r)$, equation (\ref{commRL}) follows from equation (\ref{actionpart}).

\end{Rq}

\section{Matrix coefficients of $\widetilde{V}(\lambda)$.}

\subsection{Definitions and further notation.}\label{def}
Let $\lambda\in P^+(\pi)$. Here we use the notation and results of subsection \ref{descr}.
By Lemma \ref{descr} the degenerate highest weight module $\widetilde{V}(\lambda)$ is endowed with a left $U(\tilde\p^-)$-module structure. Denote by $\widetilde{V}(\lambda)^*$ its dual vector space. Let $\tilde v\in\widetilde{V}(\lambda)$, $\tilde\xi\in\widetilde{V}(\lambda)^*$, $u\in U(\tilde\p^-)$.
Denoting by $u.\tilde v$ the action of $u$ on $\tilde v$ for this left $U(\tilde\p^-)$-module structure  on $\widetilde{V}(\lambda)$,
we denote by $\tilde\xi.u$ the right action it implies on $\widetilde{V}(\lambda)^*$, namely
  $(\tilde\xi.u)(\tilde v)=\tilde\xi(u.\tilde v)$.
Recall $\tilde{v}_{\lambda}=gr_0(v_{\lambda})$ and the isomorphism $\beta_{\lambda}:\widetilde{V}(\lambda)\longrightarrow V(\lambda)$  which is an isomorphism of left $U(\mathfrak r)$-modules. In particular this isomorphism preserves the weights. Its dual map $^t\beta_{\lambda}:V(\lambda)^*\longrightarrow \widetilde{V}(\lambda)^*$ is also an isomorphism of right $U(\mathfrak r)$-modules. By definition of $\beta_{\lambda}$, one has that $\beta_{\lambda}(\tilde v_{\lambda})=v_{\lambda}$ and then $\beta_{\lambda}(\widetilde{V'}(\lambda))=\beta(U(\mathfrak r).\tilde v_{\lambda})=U(\mathfrak r).v_{\lambda}=V'(\lambda)$. 

Recall $\tilde \n^-=\n^-_{\pi'}\ltimes(\m^-)^a\subset\tilde\p^-$.
 Set $\tilde{\xi}_{\lambda}={^t\beta_{\lambda}}(\xi_{\lambda})$ where $\xi_{\lambda}$ is the unique vector in $V(\lambda)^*$ of (right) weight $\lambda$ such that $\xi_{\lambda}(v_{\lambda})=1$. Then $\tilde{\xi}_{\lambda}$ is the unique vector in $\widetilde{V}(\lambda)^*$ of weight $\lambda$ such that $\tilde{\xi}_{\lambda}(\tilde{v}_{\lambda})=1$ and by weight considerations one has that, for all $y\in\tilde\n^-$, $\tilde\xi_{\lambda}.y=0$ since $\xi_{\lambda}$ is a vector of lowest weight in $V(\lambda)^*$. Recall that $v_{w_0\lambda}$ is a chosen nonzero vector in $V(\lambda)$ of weight $w_0\lambda$ and that it is a lowest weight vector in $V(\lambda)$. Then set $\xi_{w_0\lambda}\in V(\lambda)^*$ the vector of (right) weight $w_0\lambda$ such that $\xi_{w_0\lambda}(v_{w_0\lambda})=1$ : $\xi_{w_0\lambda}$ is a highest weight vector in $V(\lambda)^*$. Set also $\tilde v_{w_0\lambda}=\beta_{\lambda}^{-1}(v_{w_0\lambda})\in\widetilde{V}(\lambda)$. By weight considerations one has that $y.\tilde v_{w_0\lambda}=0$ for all $y\in\tilde\n^-$. Set also $\tilde\xi_{w_0\lambda}={^t\beta_{\lambda}}(\xi_{w_0\lambda})\in\widetilde{V}(\lambda)^*$. Then $\tilde\xi_{w_0\lambda}(\tilde v_{w_0\lambda})=1$. 
Since all vectors in $V(\lambda)$ of weight $w_0\lambda$ are proportional, one may observe that there exists $k_0\in\mathbb N$ such that 
 \begin{equation}v_{w_0\lambda}\in \mathscr F^{k_0}(V(\lambda)).\label{k0}\end{equation}  Then one has that 
 \begin{equation}\tilde v_{w_0\lambda}=gr_{k_0}(v_{w_0\lambda}).\label{gr}\end{equation}

 Set $V''(\lambda)=U(\mathfrak r).v_{w_0\lambda}=U(\n_{\pi'}).v_{w_0\lambda}$ : it is an irreducible $U(\mathfrak r)$-module of lowest weight $w_0\lambda$. Setting $\widetilde{V''}(\lambda)=\beta_{\lambda}^{-1}(V''(\lambda))\subset\widetilde{V}(\lambda)$, we have that $\widetilde{V''}(\lambda)=U(\mathfrak r).\tilde v_{w_0\lambda}$.
Then its dual space $\widetilde{V''}(\lambda)^*$ is such that $\widetilde{ V''}(\lambda)^*={^t\beta_{\lambda}}(V''(\lambda)^*)=\tilde\xi_{w_0\lambda}.U(\n^-_{\pi'})=\tilde\xi_{w_0\lambda}.U(\mathfrak r)\subset\widetilde{V}(\lambda)^*$.  

Since $U(\tilde\p^-)$ is a representation in $\widetilde{V}(\lambda)$ by $(i)$ of Lemma \ref{descr}, we may consider by \cite[2.7.8]{D} the space $C(\widetilde{V}(\lambda))$ of matrix coefficients of $\widetilde{V}(\lambda)$ which is the $\Bbbk$-vector subspace of $U(\tilde\p^-)^*$ generated by the linear forms $c^{\lambda}_{\xi,\,v}$ or simply $c_{\xi,\,v}$ defined by $$c_{\xi,\,v}:\;u\in U(\tilde\p^-)\mapsto \xi(u.v)\in\Bbbk$$ for all $\xi\in\widetilde{V}(\lambda)^*$ and $v\in\widetilde{V}(\lambda)$.
By equation (\ref{V}) of Lemma \ref{descr}, we may also define the $\Bbbk$-vector subspace of $C(\widetilde{V}(\lambda))$ generated by the matrix coefficients $c_{\xi,\,v'}$ with $\xi\in\widetilde{V}(\lambda)^*$ and $ v'\in\widetilde{V'}(\lambda)\subset\widetilde{V}(\lambda)$, which we will denote by $\widetilde C_{\p}(\lambda)$.

Finally denote  by $\widetilde C_{\mathfrak r}(\lambda)$ the subspace of $\widetilde C_{\p}(\lambda)$ generated by the matrix coefficients $c_{\xi,\,v}$ where $\xi\in\widetilde{V''}(\lambda)^*$ and $v\in\widetilde{V'}(\lambda)$.

\subsection{Tensor decomposition.}\label{tensor}

\begin{lm}
Let  $\lambda,\,\mu\in P^+(\pi)$. Then $\widetilde{V'}(\lambda)\otimes\widetilde{V'}(\mu)$, resp. $\widetilde {V''}(\lambda)^*\otimes\widetilde {V''}(\mu)^*$,  is a direct sum of some copies of $\widetilde{V'}(\nu)$, resp. $\widetilde{ V''}(\nu)^*$, for $\nu\in P^+(\pi)$.
 Each of them contains the unique copy of $\widetilde{V'}(\lambda+\mu)$, resp. of $\widetilde{V''}(\lambda+\mu)^*$.
\end{lm}

\begin{proof}

The proof is similar as  the proof of \cite[lemma 2.2]{FJ1}. We give it for the reader's convenience.
Let $\nu$ be an  $\h$-weight of $\widetilde{V'}(\lambda)$. Then $\nu\in\lambda-\mathbb N\pi'$. Since $\lambda\in P^+(\pi)$ we have that, for all $\alpha\in\pi\setminus\pi'$, $\langle\alpha\check\null,\,\nu\rangle\in\mathbb N$. Moreover every vector in $\widetilde{V'}(\lambda)$ is annihilated by $\rho_{\lambda}(\m)$ by remark \ref{anulm}. Since $\mathfrak r'$ is a semisimple Lie algebra, the finite dimensional $U(\mathfrak r')$-module (for diagonal action) $\widetilde{V'}(\lambda)\otimes\widetilde{V'}(\mu)$ decomposes into a direct sum of irreducible $U(\mathfrak r')$-modules, each of them being generated by  a highest weight nonzero vector whose $\h$-weight actually belongs to $P^+(\pi)$ : this highest weight nonzero vector $\tilde v$ is indeed such that $(\rho_{\lambda}\otimes\rho_{\mu})(x)(\tilde v)=0$ for all $x\in\n$, where $\rho_{\lambda}\otimes\rho_{\mu}$ is the tensor product of the representations $\rho_{\lambda}$ and $\rho_{\mu}$ as defined for instance in \cite[1.2.14]{D} and its $\h$-weight belongs to $\lambda+\mu-\mathbb N\pi'$.
Thus the tensor product $\widetilde{V'}(\lambda)\otimes\widetilde{V'}(\mu)$ is a direct sum of some copies of $\widetilde{V'}(\nu)$, for $\nu\in P^+(\pi)$ such that $\nu\in\lambda+\mu-\mathbb N\pi'$. Moreover $U(\mathfrak r).(\widetilde{v}_{\lambda}\otimes\widetilde{v}_{\mu})$ is the unique copy of $\widetilde{V}'(\lambda+\mu)$ which occurs in this tensor product.

Observe that $\widetilde{V}(\lambda)^*$ may be viewed as a left $U(\mathfrak r)$-module, by setting for all $u\in U(\mathfrak r)$, for all $\xi\in\widetilde{V}(\lambda)^*$, $u.\xi=\xi.u^{\top}$. Then $\widetilde{V''}(\lambda)^*\simeq\widetilde{V'}(-w_0\lambda)$ as left $U(\mathfrak r)$-modules. Then one obtains similarly the second part of the lemma, $\widetilde {V''}(\lambda)^*\otimes\widetilde {V''}(\mu)^*$ being a direct sum of some copies of $\widetilde{V''}(\nu)^*$, with $\nu\in P^+(\pi)$ such that $w_0\nu\in w_0\lambda+w_0\mu+\mathbb N\pi'$.
Finally $(\widetilde{\xi}_{w_0\lambda}\otimes\widetilde{\xi}_{w_0\mu}).U(\mathfrak r)$ is the unique copy of $\widetilde{V''}(\lambda+\mu)^*$ which occurs in the tensor product $\widetilde {V''}(\lambda)^*\otimes\widetilde {V''}(\mu)^*$. \end{proof}

\subsection{Direct sums.}\label{C}

Recall that the dual vector space $U(\tilde\p^-)^*$ of $U(\tilde\p^-)$ is an associative algebra with product given by the dual map of  the coproduct in $U(\tilde\p^-)$ (see for instance \cite[2.7.4]{D}).
\begin{lm}
The sum $\widetilde{C}_{\p}=\sum_{\lambda\in P^+(\pi)}\widetilde {C}_{\p}(\lambda)$ is a direct sum. The same holds for ${\widetilde C}_{\mathfrak r}=\sum_{\lambda\in P^+(\pi)}{\widetilde C}_{\mathfrak r}(\lambda)$. Moreover the latter is a subalgebra of $U(\tilde\p^-)^*$. 
\end{lm}

\begin{proof}
Let $\lambda\in P^+(\pi)$ and for all $\xi\in\widetilde{V}(\lambda)^*$, denote by $h_{\xi}:\widetilde{V'}(\lambda)\longrightarrow U(\tilde\p^-)^*$ the map such that $h_{\xi}(v)=c_{\xi,\,v}$ for all $v\in\widetilde{V'}(\lambda)$. 
For all $u'\in U(\mathfrak r)$, recall $R(u')$ the (regular) right action of $u'$ on $U(\tilde\p^-)$ defined by $R(u')(u)=uu'$ for all $u\in U(\tilde\p^-)$, where $uu'$ is the product  in $U(\tilde\p^-)$ (see subsection \ref{lractionsA}). Then its dual map  $^t R(u'):U(\tilde\p^-)^*\longrightarrow U(\tilde\p^-)^*$ defines a left action on $U(\tilde\p^-)^*$, called the coregular right representation of $U(\mathfrak r)$ on $U(\tilde\p^-)^*$ (see \cite[2.7.7]{D}). One sees easily that $h_{\xi}$ is a morphism of $U(\mathfrak r)$-modules, when $U(\tilde\p^-)^*$ is endowed with the coregular right representation of $U(\mathfrak r)$ (see also \cite[2.7.11]{D}). When $\xi\neq 0$, one checks that $h_{\xi}(\widetilde{V'}(\lambda))\neq \{0\}$ and then $h_{\xi}(\widetilde{V'}(\lambda))$ is an irreducible $U(\mathfrak r)$-module for the coregular right representation. Then $\widetilde C_{\p}(\lambda)=\sum_{\xi\in\widetilde{V}(\lambda)^*\setminus\{0\}}h_{\xi}(\widetilde {V'}(\lambda))$ is a sum of irreducible $U(\mathfrak r)$-modules all isomorphic to $\widetilde {V'}(\lambda)$. Since, for $\lambda\neq \,\mu\in P^+(\pi)$, $\widetilde {V'}(\lambda)$ and $\widetilde {V'}(\mu)$ are not isomorphic as $U(\mathfrak r)$-modules, it follows that $\sum_{\lambda\in P^+(\pi)}\widetilde C_{\p}(\lambda)$ is a direct sum. Obviously ${\widetilde C}_{\mathfrak r}=\sum_{\lambda\in P^+(\pi)}{\widetilde C}_{\mathfrak r}(\lambda)$ is also a direct sum, since ${\widetilde C}_{\mathfrak r}(\lambda)\subset\widetilde C_{\p}(\lambda)$ for all $\lambda\in P^+(\pi)$. Finally  ${\widetilde C}_{\mathfrak r}$ is an algebra by \cite[2.7.10]{D}, as a consequence of lemma \ref{tensor}.\end{proof}

\subsection{Isomorphisms.}\label{isomor}
Let $\lambda\in P^+(\pi)$ and set $\Phi_{\p}^{\lambda}:\widetilde{V}(\lambda)^*\otimes\widetilde{V'}(\lambda)\longrightarrow\widetilde{C}_{\p}(\lambda)$ defined by $\xi\otimes v'\mapsto c_{\xi,\,v'}$ and extended by linearity.
Similarly one sets $\Phi_{\mathfrak r}^{\lambda}:\widetilde{V^{''}}(\lambda)^*\otimes\widetilde{V'}(\lambda)\longrightarrow{\widetilde C}_{\mathfrak r}(\lambda)$ defined by $\xi\otimes v'\mapsto c_{\xi,\,v'}$ extended by linearity.

\begin{lm}
The maps $\Phi_{\p}^{\lambda}$ and $\Phi_{\mathfrak r}^{\lambda}$ are isomorphisms of vector spaces.

\end{lm}

\begin{proof}

Firstly these maps are obviously well defined and onto.
It remains to verify the injectivity. Assume that there exists $I$ a finite set, and for all $i\in I$, $\xi_i\in\widetilde{V}(\lambda)^*$ and $v'_i\in\widetilde{V'}(\lambda)=U(\mathfrak r).\tilde v_{\lambda}$ such that $\sum_{i\in I}c_{\xi_i,\,v'_i}=0$. We can also assume that the $v'_i$, $i\in I$, are linearly independent. We want to show that for all $i\in I$, $\xi_i=0$. Assume that there exists $i_0\in I$ such that $\xi_{i_0}\neq 0$ and complete $\xi_{i_0}$ in a basis of $\widetilde{V}(\lambda)^*$.  By taking the dual basis, there exists $v_{i_0}\in\widetilde{V}(\lambda)$ such that $\xi_{i_0}(v_{i_0})=1$. By $(i)$ of lemma \ref{descr} there exists $u_0\in U(\tilde\p^-)$ such that $v_{i_0}=u_0.\tilde v_{\lambda}$.
Now recall that $\widetilde{V'}(\lambda)$ is a left irreducible $U(\mathfrak r)$-module. Then by Jacobson density theorem (see \cite[Chap. 3,§ 3, 2]{Re}), there exists $a\in U(\mathfrak r)$ such that for all $i\in I\setminus\{i_0\}$, $a.v'_i=0$ and $a.v'_{i_0}=\tilde v_{\lambda}$. Since $u_0a\in U(\tilde\p^-)$ 
we obtain that $\sum_{i\in I} c_{\xi_i,\,v'_i}(u_0a)=0=\xi_{i_0}(u_0.(a.v'_{i_0}))=\xi_{i_0}(u_0.\tilde v_{\lambda})=\xi_{i_0}(v_{i_0})=1$ which is a contradiction. Hence the lemma for the map $\Phi_{\p}^{\lambda}$ and of course also for $\Phi_{\mathfrak r}^{\lambda}$.
\end{proof}

\subsection{The dual representation of $ad^{**}$.}\label{dualrepresentation}

Recall the left representation of $A$ in $U(\tilde\p^-)$ defined in subsection \ref{adjaction} we have denoted by $ad^{**}$. Then the dual representation of $A$ in $U(\tilde\p^-)^*$ is defined as follows.
\begin{equation}\forall a\in A,\,\forall f\in U(\tilde\p^-)^*,\;\;a.f=f\circ ad^{**} a^{\top}\label{actiondual}\end{equation}
(where recall $a^{\top}$ was defined in equation (\ref{antipode})).

This defines a left action of $A$ on $U(\tilde\p^-)^*$ (see for instance \cite[2.2.19]{D}) and by proposition \ref{lractionsA}, we deduce that one has
\begin{equation}\forall a\in A,\,\forall f\in U(\tilde\p^-)^*,\;\;a.f=f\circ L(a_2^{\top})\circ R(a_1)=({^t R}(a_1)\circ {^t L}(a_2^{\top}))(f)\label{actiondualbis}\end{equation}
where $\Delta_A(a)=a_1\otimes a_2$.

In particular one has that
\begin{equation}\forall x\in\p,\;\forall f\in U(\tilde\p^-)^*,\;x.f={^t R(x)}(f)-{^t L(x)}(f).\label{actionpsurdual}\end{equation}

Then one deduces the following lemma.
\begin{lm}
Let $\lambda\in P^+(\pi)$. One has that 
\begin{equation}\forall x\in\mathfrak r,\;\forall \xi\in\widetilde{V}(\lambda)^*,\;\forall v\in\widetilde{V'}(\lambda),\;x.c_{\xi,\,v}=c_{\xi,\,x.v}-c_{\xi.x,\,v}.\label{actionC}\end{equation}
\end{lm}
\begin{proof} Let $x\in\mathfrak r,\; \xi\in\widetilde{V}(\lambda)^*,\;v\in\widetilde{V'}(\lambda)$.
One checks easily that ${^t R(x)}(c_{\xi,\,v})=c_{\xi,\,x.v}$, resp. that ${^t L(x)}(c_{\xi,\,v})=c_{\xi.x,\,v}$, by definition of $R$, resp. of $L$, given in equation (\ref{ractionr}), resp.  in equation (\ref{lactionr}). Then the lemma follows from equation (\ref{actionpsurdual}).
\end{proof}

Let $\lambda\in P^+(\pi)$.
Endow $U(\tilde\p^-)^*$ with the dual representation of $A$ given by equation (\ref{actiondual}) and in particular with the dual representation of $U(\mathfrak r)\subset A$, which coincides in the latter case with the coadjoint representation of $U(\mathfrak r)$. By equation (\ref{actionC}) every $\widetilde{C}_{\p}(\lambda)$, resp. $\widetilde{C}_{\mathfrak r}(\lambda)$, is a left $U(\mathfrak r)$-module for the coadjoint representation.

On the other hand, endow  $\widetilde{V}(\lambda)$ with the left action of $U(\mathfrak r)$ described in subsection \ref{descr} and 
$\widetilde{V}(\lambda)^*$ with the left action of $U(\mathfrak r)$ corresponding with its dual representation, namely for all $\xi\in\widetilde{V}(\lambda)^*$, for all $u\in U(\mathfrak r)$, $u.\xi=\xi.u^{\top}$. Then endow the tensor product $\widetilde{V}(\lambda)^*\otimes\widetilde{V'}(\lambda)$ with the diagonal action of $U(\mathfrak r)$, namely for $u\in U(\mathfrak r)$ such that $\Delta(u)=u_1\otimes u_2$, for all $\xi\in\widetilde{V}(\lambda)^*$, for all $v'\in\widetilde{V'}(\lambda)$,
\begin{equation}u.(\xi\otimes v')=u_1.\xi\otimes u_2.v'=\xi.u_1^{\top}\otimes u_2.v'\label{diagaction}.\end{equation}

In particular, one has that
\begin{equation}\forall x\in\mathfrak r,\,\forall \xi\in\widetilde{V}(\lambda)^*,\,\forall v\in\widetilde{V'}(\lambda),\,x.(\xi\otimes v)=-\xi.x\otimes v+\xi\otimes x.v.\label{diagactionr}\end{equation}

Recall the isomorphisms of vector spaces $\Phi_{\p}^{\lambda}$ and $\Phi_{\mathfrak r}^{\lambda}$ defined in subsection \ref{isomor}.

\begin{prop}
Let $\lambda\in P^+(\pi)$.
With the left actions of $U(\mathfrak r)$ given by equation (\ref{diagaction}) and equation (\ref{actiondual}) respectively, the isomorphisms of vector spaces $\Phi_{\p}^{\lambda}$ and $\Phi_{\mathfrak r}^{\lambda}$ are isomorphisms of left $U(\mathfrak r)$-modules.

\end{prop}

\begin{proof}
It is immediate by equations (\ref{actionC}) and (\ref{diagactionr}).
\end{proof}

\section{${\widetilde C}_{\mathfrak r}^{U(\mathfrak r')}$ is a polynomial algebra.}

\subsection{The semigroup $\mathscr D$.}\label{polynomial}

Recall $\mathfrak r'$  the derived subalgebra of $\mathfrak r$ : the former is a semi-simple Lie algebra. Denote by $(U(\tilde\p^-)^*)^{U(\mathfrak r')}$ the set of elements in $U(\tilde\p^-)^*$ which are invariant under the coadjoint action of $U(\mathfrak r')$. Since for all $z\in\mathfrak r'$, for all $u\in U(\tilde\p^-)$ such that $\Delta(u)=u_1\otimes u_2$, we have that $\Delta(ad\,z(u))=ad\,z(u_1)\otimes u_2+u_1\otimes ad\,z(u_2)$, the set $(U(\tilde\p^-)^*)^{U(\mathfrak r')}$  is an algebra.

For all $\lambda\in P^+(\pi)$, recall that ${\widetilde C}_{\mathfrak r}(\lambda)$ is a left $U(\mathfrak r')$-module (for the coadjoint representation of $U(\mathfrak r')$) by equation (\ref{actionC}). Then define ${\widetilde C}_{\mathfrak r}(\lambda)^{U(\mathfrak r')}$ as the set of elements in ${\widetilde C}_{\mathfrak r}(\lambda)$ which are invariant under the coadjoint action of $U(\mathfrak r')$ : this is of course a vector space.

Denote   by ${\widetilde C}_{\mathfrak r}^{U(\mathfrak r')}\subset (U(\tilde\p^-)^*)^{U(\mathfrak r')}$ the set of elements in ${\widetilde C}_{\mathfrak r}$ which are invariant under the coadjoint action of $U(\mathfrak r')$. Since $\widetilde{C}_{\mathfrak r}$ is an algebra by lemma \ref{C} and  by what we said above, ${\widetilde C}_{\mathfrak r}^{U(\mathfrak r')}$ is an algebra.

Since moreover the sum of the ${\widetilde C}_{\mathfrak r}(\lambda)$'s is a direct sum by lemma \ref{C}, we have that
\begin{equation}{\widetilde C}_{\mathfrak r}^{U(\mathfrak r')}=\bigoplus_{\lambda\in P^+(\pi)}{\widetilde C}_{\mathfrak r}(\lambda)^{U(\mathfrak r')}.\label{decompC}\end{equation}

Let $\mathscr D$ be the set of all $\lambda\in P^+(\pi)$ such that 
\begin{equation}(w_0'\lambda-w_0\lambda,\,\pi')=0.\label{DD}\end{equation}

\begin{prop}

One has that, for all $\lambda\in P^+(\pi)$, $\dim {\widetilde C}_{\mathfrak r}(\lambda)^{U(\mathfrak r')}\le 1$ with equality if and only if $\lambda\in\mathscr D$ and then
\begin{equation}{\widetilde C}_{\mathfrak r}^{U(\mathfrak r')}=\bigoplus_{\lambda\in\mathscr D}{\widetilde C}_{\mathfrak r}(\lambda)^{U(\mathfrak r')}.\end{equation}

\end{prop}

\begin{proof}
The proof is quite similar as the proof in \cite[Thm. \S 3]{F2}. We give it for the reader's convenience. Fix $\lambda\in P^+(\pi)$.
Denote by ${\rm Hom}(\widetilde{V'}(\lambda)^*,\,\widetilde {V^{''}}(\lambda)^*)$ the set of all morphisms between the vector spaces $\widetilde{V'}(\lambda)^*$ and $\widetilde{V''}(\lambda)^*$, endowed with the $U(\mathfrak r')$-module structure given by 
\begin{multline}\forall u\in U(\mathfrak r'),\,\forall \varphi\in{\rm Hom}(\widetilde{V'}(\lambda)^*,\,\widetilde {V^{''}}(\lambda)^*),\,\forall\xi\in\widetilde{V'}(\lambda)^*,\\
(u.\varphi)(\xi)=u_2.\varphi(u_1^{\top}.\xi)\label{actionhom}\end{multline}
where $\Delta(u)=u_1\otimes u_2$.
Then (see for instance \cite[A.2.16]{J2}) the morphism \begin{equation}\Phi:\;\xi\otimes v'\in\widetilde {V^{''}}(\lambda)^*\otimes \widetilde{V'}(\lambda)\mapsto (\xi'\in\widetilde{V'}(\lambda)^*\mapsto \xi'(v')\xi)\label{iso}\end{equation} is an isomorphism of $U(\mathfrak r')$-modules between $\widetilde {V^{''}}(\lambda)^*\otimes \widetilde{V'}(\lambda)$, endowed with the diagonal action of $U(\mathfrak r')$ given by equation (\ref{diagaction}) and 
${\rm Hom}(\widetilde{V'}(\lambda)^*,\,\widetilde {V^{''}}(\lambda)^*)$, endowed with the action given by equation (\ref{actionhom}).

Denote by
${\rm Hom}_{U(\mathfrak r')}(\widetilde{V'}(\lambda)^*,\,\widetilde {V^{''}}(\lambda)^*)$ the set of all  $U(\mathfrak r')$-morphisms between $\widetilde{V'}(\lambda)^*$ and $\widetilde{V''}(\lambda)^*$ and by $\bigl(\widetilde{ V^{''}}(\lambda)^*\otimes \widetilde{V'}(\lambda)\bigr)^{U(\mathfrak r')}$ the set of elements in the tensor product $\widetilde{ V^{''}}(\lambda)^*\otimes \widetilde{V'}(\lambda)$ which are invariant under the diagonal action of $U(\mathfrak r')$ given by equation (\ref{diagaction}).
Then we have 
\begin{equation}\Phi(\bigl(\widetilde {V^{''}}(\lambda)^*\otimes \widetilde{V'}(\lambda)\bigr)^{U(\mathfrak r')})={\rm Hom}_{U(\mathfrak r')}(\widetilde{V'}(\lambda)^*,\,\widetilde {V^{''}}(\lambda)^*).\label{eq}\end{equation}

Moreover  the $U(\mathfrak r')$-modules $\widetilde{ V^{''}}(\lambda)^*$ and $ \widetilde{V'}(\lambda)$ (and also $\widetilde{V'}(\lambda)^*$) are irreducible. 
Then by Schur lemma (see for instance \cite[Chap. 3,\,\S\,3, \,1]{Re}), \begin{equation}\dim{\rm Hom}_{U(\mathfrak r')}(\widetilde{V'}(\lambda)^*,\,\widetilde {V^{''}}(\lambda)^*)\le 1\label{dim}\end{equation} with equality if and only if the irreducible $U(\mathfrak r')$-modules $\widetilde{V^{''}}(\lambda)^*$ and $\widetilde{V'}(\lambda)^*$ are isomorphic that is, if and only if \begin{equation}w_0'\lambda-w_0\lambda=\sum_{\alpha\in\pi\setminus\pi'}m_{\alpha}\varpi_{\alpha}\;\hbox{\rm with}\;m_{\alpha}\in\mathbb N,\,\forall\alpha\in\pi\setminus\pi'\label{eqD}\end{equation}
or equivalently if and only if $\lambda$ verifies equation (\ref{DD}) that is, if and only if $\lambda\in\mathscr D$.

Indeed we have that $\widetilde{V''}(\lambda)^*\simeq \widetilde{V'}(-w_0\lambda)$ as  left $U(\mathfrak r)$-modules by what we already said in the proof of Lemma \ref{tensor}. Similarly since $\widetilde{V'}(\lambda)=U(\n^-_{\pi'}).\tilde v_{\lambda}=U(\n_{\pi'}).\tilde v_{w_0'\lambda}$ where $\tilde v_{w_0'\lambda}$ is a chosen nonzero weight vector in $\widetilde{V'}(\lambda)$ of weight $w_0'\lambda$, we have that $\widetilde{V'}(\lambda)^*\simeq \widetilde{V'}(-w'_0\lambda)$ as left $U(\mathfrak r)$-modules. Then the irreducible $U(\mathfrak r')$-modules $\widetilde{V^{''}}(\lambda)^*$ and $\widetilde{V'}(\lambda)^*$ are isomorphic if and only if $(-w_0\lambda)'=(-w'_0\lambda)'$ where recall that the superscript ``prime''  denotes the projection in $P(\pi')$ of an element in $P(\pi)$ with respect to the decomposition (\ref{projP}).

By proposition \ref{dualrepresentation} one has that \begin{equation}{\widetilde C}_{\mathfrak r}(\lambda)^{U(\mathfrak r')}=\Phi_{\mathfrak r}^{\lambda}\Bigl(\bigl(\widetilde{ V^{''}}(\lambda)^*\otimes \widetilde{V'}(\lambda)\bigr)^{U(\mathfrak r')}\Bigr).\label{decompsum}\end{equation}
Then by equations (\ref{eq}) and (\ref{dim}) we have that $\dim{\widetilde C}_{\mathfrak r}(\lambda)^{U(\mathfrak r')}\le 1$ with equality if and only if $\lambda\in\mathscr D$.
For all $\lambda\in\mathscr D$ set $c_{\lambda}\in{\widetilde C}_{\mathfrak r}(\lambda)^{U(\mathfrak r')}\setminus\{0\}$ so that ${\widetilde C}_{\mathfrak r}(\lambda)^{U(\mathfrak r')}=\Bbbk c_{\lambda}$. By equation (\ref{decompC}) we also have
\begin{equation} {\widetilde C}_{\mathfrak r}^{U(\mathfrak r')}=\bigoplus_{\lambda\in\mathscr D}\Bbbk c_{\lambda}.\label{sumC}\end{equation}
This completes the proof.\end{proof}

Let $\lambda\in\mathscr D$. Choose $\varphi_{\lambda}\in{\rm Hom}_{U(\mathfrak r')}(\widetilde{V'}(\lambda)^*,\,\widetilde {V^{''}}(\lambda)^*)\setminus\{0\}$ and denote by $U(\mathfrak r')_+$ the kernel of the coidentity in the enveloping algebra $U(\mathfrak r')$. By \cite[7.1.16]{J2} we have that 
\begin{equation}\Phi^{-1}:\,{\rm Hom}(\widetilde{V'}(\lambda)^*,\,\widetilde {V^{''}}(\lambda)^*)\stackrel{\sim}{\longrightarrow} U(\mathfrak r')_+.(\tilde\xi_{w_0\lambda}\otimes\tilde v_{w_0'\lambda})\oplus \Bbbk \Phi^{-1}(\varphi_{\lambda}).\end{equation}
 It follows that we have, up to a nonzero scalar 
\begin{equation}({\Phi_{\mathfrak r}^{\lambda}})^{-1}(c_{\lambda})=\tilde\xi_{w_0\lambda}\otimes\tilde v_{w_0'\lambda}+\sum_{i\in I}u^-_i.\tilde\xi_{w_0\lambda}\otimes u^+_i.\tilde v_{w_0'\lambda}\label{descrinvariant}\end{equation}
where $u^{\pm}_i\in\n^{\pm}_{\pi'}U(\n_{\pi'}^{\pm})$ for all $i\in I$, $I$ a finite set, since moreover $\tilde\xi_{w_0\lambda}\otimes\tilde v_{w_0'\lambda}$ is a cyclic vector for the $U(\mathfrak r')$-module $\widetilde{ V^{''}}(\lambda)^*\otimes \widetilde{V'}(\lambda)$ endowed with the diagonal action. Hence the $\h$-weight of $c_{\lambda}$ is equal to \begin{equation}w_0'\lambda-w_0\lambda.\label{weightgenerator}\end{equation}

By equation (\ref{eqD}) this weight belongs to $P^+(\pi)$ and by equation (\ref{DD}) it annihilates on $\pi'$.

Let $i$ and $j$ denote the permutations in $\pi$ defined below. 
\begin{equation}\forall\alpha\in\pi,\,j(\alpha)=-w_0(\alpha)\end{equation}  
\begin{equation}\forall\alpha\in\pi', \,i(\alpha)=-w_0'(\alpha)\end{equation}
 \begin{equation}\forall\alpha\in\pi\setminus\pi',\,\begin{cases}i(\alpha)=j(\alpha)& {\rm if}\, j(\alpha)\not\in\pi'\\
  i(\alpha)=j(ij)^{r_{\alpha}}(\alpha)&{\rm otherwise}\end{cases}\end{equation} where $r_{\alpha}$ is the smallest integer such that $j(ij)^{r_{\alpha}}(\alpha)\not\in\pi'$. Let $E(\pi')$ be the set of $\langle ij\rangle$-orbits in $\pi$, where $\langle ij\rangle$ denotes the subgroup generated by the composition map $ij$. 

For instance, if $\p$ is a maximal parabolic subalgebra of $\g$ that is, if $\pi\setminus\pi'=\{\alpha\}$, then $i(\alpha)=\alpha$ and the $\langle ij\rangle$-orbit of $\alpha$ is $\Gamma_{\alpha}=\{(ji)^sj(\alpha),\;0\le s\le r_{\alpha}\}$.

Recall \cite[Thm. 1]{FJ1} (see also \cite[4.1]{F0}):
\begin{thm}
The set $\mathscr D$ is a free additive semigroup  generated by the $\mathbb Z$-linearly independent elements $d_{\Gamma}=\sum_{\gamma\in\Gamma}\varpi_{\gamma}$, $\Gamma\in E(\pi')$.
\end{thm}

\subsection{A filtration on the algebra $\widetilde C_{\mathfrak r}$.}\label{filtrationC}

Assume that $\pi=\{\alpha_1,\,\ldots,\,\alpha_n\}$. Then  for all $\lambda\in P^+(\pi)$, there exist $k_i\in\mathbb Q_+$ for all $i$, $1\le i\le n$, such that $\lambda=\sum_{i=1}^nk_i\alpha_i$. Set $deg(\lambda)=2\sum_{i=1}^nk_i$. By \cite[7.1.25]{J2}, $deg(\lambda)\in\mathbb N$. For all $m\in\mathbb N$, we set $\mathscr F'_m(\widetilde C_{\mathfrak r})=\bigoplus_{\lambda\in P^+(\pi)\mid deg(\lambda)\le m}\widetilde C_{\mathfrak r}(\lambda)$, which is a left $U(\mathfrak r)$-submodule of $\widetilde C_{\mathfrak r}$ for coadjoint action. Then $(\mathscr F'_m(\widetilde C_{\mathfrak r}))_{m\in\mathbb N}$ is an increasing filtration $\mathscr F'$ of the algebra $\widetilde C_{\mathfrak r}$ since for all $\lambda,\,\mu\in P^+(\pi)$, \begin{equation}\widetilde C_{\mathfrak r}(\lambda)\widetilde C_{\mathfrak r}(\mu)\subset \bigoplus_{\nu\in\mathbb N\pi'\mid\lambda+\mu-\nu\in P^+(\pi)}\widetilde C_{\mathfrak r}(\lambda+\mu-\nu)\label{decomp}\end{equation} by the proof of lemma \ref{tensor}.
Then denote by $gr_{\mathscr F'}(\widetilde C_{\mathfrak r})$ the associated graded algebra and for all $c\in\mathscr F'_m(\widetilde C_{\mathfrak r})$, denote by $gr_{m,\,\mathscr F'}(c)$ its canonical image in $gr_{\mathscr F'}(\widetilde C_{\mathfrak r})$. Recall the notation in subsection \ref{polynomial}.

\begin{lm}

Let $\lambda,\,\mu\in\mathscr D$. Set $m=deg(\lambda)$ and $m'=deg(\mu)$. 

Then 
$gr_{m,\,\mathscr F'}(c_{\lambda})gr_{m',\,\mathscr F'}(c_{\mu})$ is a nonzero multiple of $gr_{m+m',\,\mathscr F'}(c_{\lambda+\mu}).$

\end{lm}

\begin{proof}
By definition of the multiplication in the graded algebra $gr_{\mathscr F'}(\widetilde{C}_{\mathfrak r})$, one has that
$$gr_{m,\,\mathscr F'}(c_{\lambda})gr_{m',\,\mathscr F'}(c_{\mu})=gr_{m+m',\,\mathscr F'}(c_{\lambda}c_{\mu}).$$
Now by equation (\ref{descrinvariant}) in the product $c_{\lambda}c_{\mu}$ appears, up to a nonzero scalar, the term $$c_{\tilde\xi_{w_0\lambda}\otimes\tilde\xi_{w_0\mu},\,\tilde v_{w_0'\lambda}\otimes \tilde v_{w_0'\mu}}=c_{\tilde\xi_{w_0(\lambda+\mu)},\,\tilde v_{w_0'(\lambda+\mu)}}\in \widetilde C_{\mathfrak r}(\lambda+\mu).$$ Indeed this term  cannot be annihilated by the other terms, by lemmas \ref{C} and \ref{isomor}.

Since moreover $c_{\lambda}c_{\mu}\in\widetilde{C}_{\mathfrak r}^{U(\mathfrak r')}$,
equations (\ref{sumC}), (\ref{descrinvariant}) and (\ref{decomp}) imply that

$$\begin{array}{cll} gr_{m+m',\,\mathscr F'}(c_{\lambda}c_{\mu})&=&\sum_{\nu\in\mathscr D,\,deg(\nu)=m+m',\,\nu\in\lambda+\mu-\mathbb N\pi'}gr_{m+m',\,\mathscr F'}(c_{\nu})\\
&=&gr_{m+m',\,\mathscr F'}(c_{\lambda+\mu})\end{array}$$ up to multiplication by a nonzero scalar.
\end{proof}

Now we can conclude the following.
\begin{prop}
The algebra of invariants ${\widetilde C}_{\mathfrak r}^{U(\mathfrak r')}$ is a polynomial algebra over $\Bbbk$, whose number of algebraically independent generators is equal to the cardinality of the set $E(\pi')$.
\end{prop}

\begin{proof}

It follows as in the proof of  \cite[Thm. 1]{FJ1} (see also  \cite[Prop. 3.1]{FJ1}). Let $\lambda_i,\,i\in I$, be a set of $\mathbb Z$-linearly independent generators of $\mathscr D$ and set $m_i=deg(\lambda_i)$ for all $i\in I$
(one has that $\lvert I\rvert=\lvert E(\pi')\rvert$ by Thm \ref{polynomial}). Denote by $gr_{\mathscr F'}({\widetilde C}_{\mathfrak r}^{U(\mathfrak r')})$ the graded algebra of the algebra ${\widetilde C}_{\mathfrak r}^{U(\mathfrak r')}$ associated with the induced filtration. Note that the above lemma also holds in this graded algebra. Then equation (\ref{sumC}) and  the above lemma imply that $gr_{m_i,\,\mathscr F'}(c_{\lambda_i}), i\in I$, are $\Bbbk$-algebraically independent and generate $gr_{\mathscr F'}({\widetilde C}_{\mathfrak r}^{U(\mathfrak r')})$.  Hence $gr_{\mathscr F'}({\widetilde C}_{\mathfrak r}^{U(\mathfrak r')})$ is a polynomial algebra over $\Bbbk$ in $\lvert E(\pi')\rvert$ generators and it follows (see \cite[Chap. III, \S\, 2, n$^{\circ}$ 9, Prop. 10]{Bou2}) that the algebra ${\widetilde C}_{\mathfrak r}^{U(\mathfrak r')}$ is also a polynomial algebra over $\Bbbk$ in $\lvert E(\pi')\rvert$ generators $c_{\lambda_i}$, $i\in I$, whose $\h$-weight  is equal to $\delta_{i}=w_0'\lambda_i-w_0\lambda_i$ by equation (\ref{weightgenerator}).
\end{proof}

\section{Generalized Kostant filtration and morphism.}\label{Kfm}

\subsection{Generalized Kostant filtration.}\label{Kf}
In \cite[6.1]{FJ2} we have defined what we called the Kostant filtration (denoted by $\mathscr F_K$) on the Hopf dual  of the enveloping algebra of the simple Lie algebra $\g$. 
Here we will consider what we call {\it the generalized Kostant filtration} on the dual algebra $U(\tilde\p^-)^*$ of $U(\tilde\p^-)$. More precisely, we set
\begin{equation}\forall k\in\mathbb N,\,\mathscr F_K^k(U(\tilde\p^-)^*)=\{f\in U(\tilde\p^-)^*\mid f(U_{k-1}(\tilde\p^-))=0\}\label{defKostfiltr}\end{equation}
where $(U_k(\tilde\p^-))_{k\in\mathbb N\cup\{-1\}}$ is the canonical filtration on $U(\tilde\p^-)$, with $U_{-1}(\tilde\p^-)=\{0\}$.

\begin{lm}
The generalized Kostant filtration $(\mathscr F_K^k(U(\tilde\p^-)^*))_{k\in\mathbb N}$ is a decreasing, exhaustive and separated filtration on the algebra $U(\tilde\p^-)^*$. Moreover this filtration is invariant by the left action of $A$ defined by equation (\ref{actiondual}).

\end{lm}
\begin{proof}

The first assertions are obvious.

Let us show the invariance by the left action of $A$.
Let $a\in A$, $k\in\mathbb N$ and $f\in\mathscr F_K^k(U(\tilde\p^-)^*)$.

If $a=z\in\mathfrak r$, then $ad^{**}z=ad\,z$ by equation (\ref{actionAr}) and for all $u\in U_{k-1}(\tilde\p^-)$, $ad\,z(u)\in U_{k-1}(\tilde\p^-)$. Then $z.f\in \mathscr F_K^k(U(\tilde\p^-)^*)$.

Now assume that $a=x\in\m$, and let $u\in U_{k-1}(\tilde\p^-)$. Recall equation (\ref{directsumbis}).
Then $u=\sum_{j=0}^{k-1}s_ju'_j$ with $s_j\in S_j(\m^-)$ and $u'_j\in U_{k-1-j}(\mathfrak r)$, for all $0\le j\le k-1$.

Then by equation (\ref{actionA}) one has :
$$ad^{**}x(u)=\sum_{j=0}^{k-1}\tilde\theta(ad^*x(s_j))u'_j\in\sum_{j=0}^{k-1}S_{j-1}(\m^-)U_{k-j}(\mathfrak r)\subset U_{k-1}(\tilde\p^-)$$
by equation (\ref{tildethetaad*}).

It follows that $x.f(u)=0$ and the lemma.
\end{proof}

\subsection{The graded algebra associated with the generalized Kostant filtration.}\label{gra}

 Set \begin{equation}gr_{K}(U(\tilde\p^-)^*)=\bigoplus_{k\in\mathbb N}gr_K^k(U(\tilde\p^-)^*)\label{gradK}\end{equation}
 where, for all $k\in\mathbb N$, \begin{equation}gr_K^k(U(\tilde\p^-)^*)=\mathscr F_K^k(U(\tilde\p^-)^*)/\mathscr F_K^{k+1}(U(\tilde\p^-)^*).\label{gradKk}\end{equation}

Then $gr_{K}(U(\tilde\p^-)^*)$ is the graded algebra associated with the generalized Kostant filtration $\mathscr F_K$ on the algebra $U(\tilde\p^-)^*$. For all $f\in\mathscr F_K^k(U(\tilde\p^-)^*)$, one denotes by $gr_K^k(f)$ its canonical image in $gr_K^k(U(\tilde\p^-)^*)$.

By lemma \ref{Kf} the dual representation of $A$ on $U(\tilde\p^-)^*$ (given by equation (\ref{actiondual})) induces a left action on $gr_{K}(U(\tilde\p^-)^*)$ defined by
\begin{equation}\forall a\in A,\forall f\in\mathscr F_K^k(U(\tilde\p^-)^*),\,a.gr_K^k(f)=gr_K^k(a.f).\label{actiongradK}\end{equation}

\begin{prop}
Let $x\in\m$ and $f\in\mathscr F_K^k(U(\tilde\p^-)^*)\cap\widetilde{C}_{\mathfrak r}$, $k\in\mathbb N$. Then $x.f\in\mathscr F_K^{k+1}(U(\tilde\p^-)^*)$ and therefore \begin{equation}x.gr_K^k(f)=0\label{annulm}\tag{$\diamond$}\end{equation}
where recall $gr_K^k(g)=g+\mathscr F_K^{k+1}(U(\tilde\p^-)^*)$, for all $g\in\mathscr F_K^k(U(\tilde\p^-)^*)$. 
\end{prop}

\begin{proof}

Take $$f=\sum_{\lambda\in\Lambda}\sum_{i\in I_{\lambda}}c^{\lambda}_{\tilde\xi_{w_0\lambda}.u'_i,\,u''_i.\tilde v_{\lambda}}\in\mathscr F_K^k(U(\tilde\p^-)^*)\cap\widetilde{C}_{\mathfrak r}$$
with $\Lambda\subset P^+(\pi)$ a finite set and for all $\lambda\in\Lambda$, $I_{\lambda}$ a finite set, $u'_i,\,u''_i\in U(\n^-_{\pi'})$, for all $i\in I_{\lambda}$. Moreover one may assume, if $f\neq 0$, that $u'_i,\,u''_i$, for all $i\in I_{\lambda}$, are nonzero weight vectors.
We need the lemma below.

\begin{lm}
Let $k\in\mathbb N$ and $0\le j\le k$.
With the above hypotheses, we have that
\begin{equation}\forall u\in U_{j-1}(\g),\,\forall u'\in U_{k-j}(\mathfrak r),\,\,\sum_{\lambda\in\Lambda}\sum_{i\in I_{\lambda}}\xi_{w_0\lambda}(u'_iuu'u''_i. v_{\lambda})=0. \label{nul}\end{equation}

\end{lm}

\begin{proof}
The lemma is obvious for $j=0$ since $U_{-1}(\g)=\{0\}$. Assume that $j\in\mathbb N^*$.
Take  $u\in U_{j-1}(\g)$ and $u'\in U_{k-j}(\mathfrak r)$.
Since $\m.V'(\lambda)=\{0\}$ and since $U_{j-1}(\g)=U_{j-1}(\p^-)\oplus U_{j-2}(\g)\m$, one may assume that $u\in U_{j-1}(\p^-)$. One also may assume  that $u$ and $u'$ are nonzero weight vectors.

Since $\tilde\xi_{w_0\lambda}$ vanishes on the weight vectors  $\beta_{\lambda}^{-1}(u'_iuu'u''_i .v_{\lambda})$ which are not of weight $w_0\lambda$ and by equation (\ref{gr}), one has that
\begin{align*}\sum_{\lambda\in\Lambda}\sum_{i\in I_{\lambda}}\xi_{w_0\lambda}(u'_iuu'u''_i .v_{\lambda})&=
\sum_{\lambda\in\Lambda}\sum_{i\in I_{\lambda}}\tilde\xi_{w_0\lambda}(\beta_{\lambda}^{-1}(u'_iuu'u''_i .v_{\lambda}))\\
&=\sum_{\lambda\in\Lambda}\sum_{i\in I'_{\lambda}}\tilde\xi_{w_0\lambda}(gr_{k_0}(u'_iuu'u''_i. v_{\lambda}))\\
\end{align*}
where for all $\lambda\in\Lambda$, $I'_{\lambda}\subset I_{\lambda}$ is the set of indices $i\in I_{\lambda}$ such that $\beta_{\lambda}^{-1}(u'_iuu'u''_i .v_{\lambda})$ is a vector of weight $w_0\lambda$.

Now write $u=\sum_{t=0}^{j-1}u_tv_t$ with $u_t=\theta(s_t)\in U^t(\m^-)=\theta(S_t(\m^-))$, $s_t\in S_t(\m^-)$ and $v_t\in U_{j-1-t}(\mathfrak r)$ for all $0\le t\le j-1$.

Then \begin{equation*}gr_{k_0}(u'_iuu'u''_i.v_{\lambda})=u'_i\Bigl(\sum_{t=0}^{j-1}s_tv_t\Bigr)u'u''_i.\tilde v_{\lambda}\end{equation*} by equation (\ref{first}).

But $\sum_{t=0}^{j-1}s_tv_t\in U_{j-1}(\tilde\p^-)$ and recall that $u'\in U_{k-j}(\mathfrak r)$. Then \begin{equation*}\Bigl(\sum_{t=0}^{j-1}s_tv_t\Bigr)u'\in U_{k-1}(\tilde\p^-)\end{equation*}
and we obtain the required equation (\ref{nul}) since $f(U_{k-1}(\tilde\p^-))=0$.
\end{proof}

%\begin{lm}

%Recall that for all $k\in\mathbb N$, we set $U^k(\m^-)=\theta(S_k(\m^-))$, where $\theta$ is the symmetrization.  For $k\in\mathbb N$ and $j\in\mathbb N$, take $u\in U^k(\m^-)\otimes U_j(\n^-_{\pi'})$
%and $u'\in U_{k-1}(\m^-)\otimes U(\n^-_{\pi'})$ such that $u-u'\in {\rm Ann}_{U(\n^-)}(v_{\lambda})$.

%\end{lm}

Fix $x\in\m$ being a nonzero weight vector and for all $0\le j\le k$, take $s\in S_j(\m^-)$ and  $u'\in U_{k-j}(\mathfrak r)$ being weight vectors.
If $j\ge 1$, one may assume   that
$s=y_1\cdots y_j\in S_j(\m^-)$ with $y_t\in\m^-$ being a weight vector for all $1\le t\le j$.  

Recall that \begin{equation}ad^*x(s)=\sum_{1\le t\le j\mid [x,\,y_t]\in\mathfrak r}y_1\cdots y_{t-1}[x,\,y_t]y_{t+1}\cdots y_j\in S_j(\p^-).\end{equation}
Set
\begin{equation}ad_{\m^-}x(s)=\sum_{1\le t\le j\mid [x,\,y_t]\in\mathfrak \m^-}y_1\cdots y_{t-1}[x,\,y_t]y_{t+1}\cdots y_j\in S_j(\m^-)\end{equation}
and
\begin{equation}ad_{\m}x(s)=\sum_{1\le t\le j\mid[x,\,y_t]\in\m}y_1\cdots y_{t-1}[x,\,y_t]y_{t+1}\cdots y_j\in S_j(\g).\end{equation}
By equations (\ref{actiondual}) and (\ref{actionA}),  one has that \begin{equation*}-x.f(su')=f(ad^{**}x(su'))=f(\tilde\theta(ad^*x(s))u').\end{equation*} Then if $j=0$ one has obviously that $x.f(su')=0$, by equation (\ref{actionnul}).

From now on, assume that $j\ge 1$. 
By the above and by what we said in the proof of the previous lemma we have that
\begin{equation}-x.f(su')
=\sum_{\lambda\in\Lambda}\sum_{i\in I'_{\lambda}}\tilde\xi_{w_0\lambda}\Bigr(u'_i\tilde\theta(ad^*x(s))u'u''_i.\tilde v_{\lambda}\Bigr)\label{eqq}\end{equation}
where for all $\lambda\in\Lambda$, $I'_{\lambda}=\{i\in I_{\lambda}\mid\exists c_i\in\Bbbk^*;\; u'_i\tilde\theta(ad^*x(s))u'u''_i.\tilde v_{\lambda}=c_i\tilde v_{w_0\lambda}\}$.

Fix $\lambda\in\Lambda$ and $i\in I'_{\lambda}$.  One has that $u'_i\tilde\theta(ad^*x(s))u'u''_i.\tilde v_{\lambda}\in gr_{j-1}(V(\lambda))$ by equations (\ref{tildethetaad*}) and (\ref{third}),
and  by equation (\ref{gr}) we have that $j-1=k_0$. 

%, we have that 
%and similarly that
Consider $\theta:S(\g)\longrightarrow U(\g)$ the symmetrization (which is an isomorphism of $ad\,U(\g)$-modules)
%Assume that $s=y_1\cdots y_j$ (product in $S(\m^-)$), with for all $1\le t\le j$, $y_t\in\m^-$ a nonzero weight vector. Let $J$ be the maximal subset of $\{1,\,\ldots,\,j\}$ such that $ad\,x(y_t)\in\mathfrak r$ for all $t\in J$. 
and the adjoint action (denoted by $ad$) of $\m$ on $S(\g)$ which extends uniquely by derivation the adjoint action of $\m$ on $\g$ given by Lie bracket.
%Set $ad\,x(s) \in S(\g)$ the adjoint action of $x$ on $s\in S(\g)$.
Observe that 
\begin{equation}ad\,x(s)=ad^*x(s)+ad_{\m^-}x(s)+ad_{\m}x(s)\label{sum}.\end{equation}

Moreover for all $1\le t\le j$, one has that
\begin{align}\theta(y_1\cdots y_{t-1}[x,\,y_t]y_{t+1}\cdots y_j)=\notag\\
\theta(y_1\cdots y_{t-1}y_{t+1}\cdots y_j)[x,\,y_t]\mod U_{j-1}(\g).\label{annull}\end{align}

By equations (\ref{eqq}), (\ref{annull}) and the previous lemma, it follows that
\begin{equation}-x.f(su')
=\sum_{\lambda\in\Lambda}\sum_{i\in I'_{\lambda}}\xi_{w_0\lambda}\Bigr(u'_i\theta(ad^*x(s))u'u''_i.v_{\lambda}\Bigr)\end{equation}

%For every $\lambda\in\Lambda$, set $$I''_{\lambda}=\{i\in I'_{\lambda}\mid\xi_{w_0\lambda}(u'_i\theta(ad^*x(s))u'u''_i.v_{\lambda})\neq 0\}$$
%so that
%\begin{equation}-x.f(su')
%=\sum_{\lambda\in\Lambda}\sum_{i\in I''_{\lambda}}\xi_{w_0\lambda}\Bigr(u'_i\theta(ad^*x(s))u'u''_i.v_{\lambda}\Bigr).\end{equation}
By equation (\ref{annull}) and the previous lemma, and since $\m.V'(\lambda)=\{0\}$  one has that
\begin{equation}\sum_{\lambda\in\Lambda}\sum_{i\in I'_{\lambda}}\xi_{w_0\lambda}(u'_i\theta(ad_{\m}x(s))u'u''_i.v_{\lambda})=0.\end{equation}

We claim that
\begin{equation}-x.f(su')=\sum_{\lambda\in\Lambda}\sum_{i\in I'_{\lambda}}\xi_{w_0\lambda}\Bigl(u'_i\theta(ad\,x(s))u'u''_i.v_{\lambda}\Bigr).\label{expressionf}
\end{equation}

By equation (\ref{sum}) it remains to show that 
\begin{equation}\sum_{\lambda\in\Lambda}\sum_{i\in I'_{\lambda}}\xi_{w_0\lambda}(u'_i\theta(ad_{\m^-}x(s))u'u''_i.v_{\lambda})=0\label{req}.\end{equation}

%Suppose the opposite and for all $\lambda\in\Lambda$, set
%$$\tilde I_{\lambda}=\{i\in I''_{\lambda}\mid\xi_{w_0\lambda}(u'_i\theta(ad_{\m^-}x(s))u'u''_i.v_{\lambda})\neq 0\}$$  so that
%\begin{equation}\sum_{\lambda\in\Lambda}\sum_{i\in I''_{\lambda}}\xi_{w_0\lambda}(u'_i\theta(ad_{\m^-}x(s))u'u''_i.v_{\lambda})=\sum_{\lambda\in\Lambda}\sum_{i\in\tilde I_{\lambda}}\xi_{w_0\lambda}(u'_i\theta(ad_{\m^-}x(s))u'u''_i.v_{\lambda}).\end{equation}

Fix an index $t$ with $[x,\,y_t]\in\m^-$. Take $\lambda\in\Lambda$ and $i\in I'_{\lambda}$ (if there exist) such that
 $$\xi_{w_0\lambda}(u'_i\theta(y_1\cdots y_{t-1}[x,\,y_t]y_{t+1}\cdots y_j)u'u''_i.v_{\lambda})\neq 0.$$ 
%Fix $\lambda\in\Lambda$ and $i\in \tilde I_{\lambda}$. %If for all $t'\neq t$ such that $[x,\,y_{t'}]\in\mathfrak r$, one has that $$\xi_{w_0\lambda}(u'_i\theta(y_1\cdots y_{t'-1}[x,\,y_{t'}]y_{t'+1}\cdots y_j)u'u''_i.v_{\lambda})=0$$

Since all weight vectors non vanishing on $\xi_{w_0\lambda}$ are proportional to $v_{w_0\lambda}$,  one has that
$u'_i\theta(y_1\cdots y_{t-1}[x,\,y_t]y_{t+1}\cdots y_j)u'u''_i.v_{\lambda}$ is proportional to $v_{w_0\lambda}$. 

On the other hand, one knows that $v_{w_0\lambda}\in\mathscr F^{k_0}(V(\lambda))\subset U^{k_0}(\m^-).V'(\lambda)$ by $(iv)$ of Lemma \ref{descr} and that $k_0=j-1$ (otherwise $I'_{\lambda}=\emptyset$). Then by the irreducibility of the $U(\mathfrak r)$-module $V'(\lambda)$,  there exists a nonzero weight vector $u'_{i,\,t}\in U^{j-1}(\m^-)U(\n^-_{\pi'})U(\n_{\pi'})$ such that \begin{equation}(u'_i\theta(y_1\cdots y_{t-1}[x,\,y_t]y_{t+1}\cdots y_j)-u'_{i,\,t}).(u'u''_i.v_{\lambda})=0.\end{equation}

Set $u_t=\theta(y_1\cdots y_{t-1}[x,\,y_t]y_{t+1}\cdots y_j)$. Then $u_{t}\in U^j(\m^-)$ is such that
\begin{equation}u'_iu_{t}-u'_{i,\,t}\in {\rm Ann}_{U(\m^-)U(\n^-_{\pi'})U(\n_{\pi'})}(u'u''_i.v_{\lambda})\end{equation}

For all $\gamma\in\Delta^+\setminus\Delta^+_{\pi'}$, denote by $r_{\gamma,\,i}$ the smallest positive integer such that $x_{-\gamma}^{r_{\gamma,\,i}}.(u'u''_i.v_{\lambda})=0$. If we denote by $\mu_i$ the weight of the vector $u'u''_i.v_{\lambda}$, we have that $r_{\gamma,\,i}=\langle\gamma\,\check\null,\,\mu_i\rangle+1$, since $x_{\gamma}.(u'u''_i.v_{\lambda})=0$.
%Since $v_{w_0\lambda}\in\mathscr F^{k_0}(V(\lambda))$, one has necessarily that $r_{\alpha,\,i}\le k_0$.

Similarly for all $\beta\in\Delta^+_{\pi'}$ denote by $r_{\beta,\,i}^\pm$ the smallest positive integer such that $x_{\pm\beta}^{r^\pm_{\beta,\,i}}.(u'u''_i.v_{\lambda})=0$ (see \cite[21]{H}).

%Denote by $\Delta'_i$ the subset of $\Delta^+\setminus\Delta^+_{\pi'}$ such that $r_{\alpha,\,i}\le j$.
Then one has that
\begin{align}u'_iu_{t}-u'_{i,\,t}&\in\sum_{\gamma\in\Delta^+\setminus\Delta^+_{\pi'}}U(\m^-)U(\n^-_{\pi'})U(\n_{\pi'})x_{-\gamma}^{r_{\gamma,\,i}}\label{form}\\
\notag&+\sum_{\beta\in\Delta^+_{\pi'}}U(\m^-)U(\n^-_{\pi'})U(\n_{\pi'})x_{\pm\beta}^{r^\pm_{\beta,\,i}}.\end{align}

By the Poincar\'e-Birkhoff-Witt theorem (\cite[2.1.11]{D})  setting  $\Delta^+\setminus\Delta^+_{\pi'}=\{\gamma_1,\,\ldots,\,\gamma_r\}$, we have that
\begin{equation} u_t=\sum_{\vec{\nu}}c_{\vec{\nu}}\,x_{-\gamma_1}^{\nu_1}\cdots x_{-\gamma_r}^{\nu_r}\label{PB}\end{equation}
where the sum over the $\vec{\nu}=(\nu_1,\,\ldots,\,\nu_r)\in\mathbb N^r$  is finite and with, for all $\vec{\nu}$, $c_{\vec{\nu}}\in\Bbbk$.

One has $U(\mathfrak r)U^j(\m^-)=U^j(\m^-)U(\mathfrak r)$ and $U^j(\m^-)U(\mathfrak r)\cap U^{j-1}(\m^-)U(\mathfrak r)=\{0\}$. 

Comparing equations (\ref{form}) and (\ref{PB}) one deduces that 
 there exists $w_{\gamma,\,i}\in U(\m^-)$, for all $\gamma\in\Delta^+\setminus\Delta^+_{\pi'}$, and that there exists  $w_{\pm\beta,\,i}\in U^j(\m^-)U(\n^-_{\pi'})U(\n_{\pi'})$, for all $\beta\in\Delta^+_{\pi'}$, such that
%\begin{equation}u'_iu_{t}-u'_{i,\,t}=\sum_{\alpha\in\Delta^+\setminus\Delta^+_{\pi'}}u'_iw_{\alpha,\,i}x_{-\alpha}^{r_{\alpha,\,i}}+\sum_{\beta\in \Delta^+_{\pi'}}\sum_{k\in K_{\pm\beta}}w_{k}w'_kx_{\pm\beta}^{r_{\beta,\,i}^\pm}.\end{equation}
%Denote by $\Delta''_i$ the subset of $\beta\in\Delta^+_{\pi'}$ such that $\sum_{\beta\in\Delta''_i}\sum_{k\in K_{\pm\beta}}w_{k}w'_kx_{\pm\beta}^{r_{\beta,\,i}^\pm}=u'_i u$ with $u\in U_j(\m^-)$.
 \begin{equation}u'_iu_{t}=\sum_{\gamma\in\Delta^+\setminus\Delta^+_{\pi'}}u'_iw_{\gamma,\,i}x_{-\gamma}^{r_{\gamma,\,i}}+\sum_{\beta\in\Delta^+_{\pi'}}w_{\pm\beta,\,i}x_{\pm\beta}^{r_{\beta,\,i}^\pm}.\end{equation}

 Then one has that
\begin{equation}\sum_{\lambda\in\Lambda}\sum_{i\in I'_{\lambda}} \xi_{w_0\lambda}(u'_i u_tu'u''_i.v_{\lambda})=0\end{equation}
and
\begin{align*}\sum_{\lambda\in\Lambda}\sum_{i\in I'_{\lambda}}\xi_{w_0\lambda}(u'_i\theta(ad_{\m^-}x(s))u'u''_i.v_{\lambda})=&\\
\sum_{1\le t\le j\mid [x,\,y_t]\in\m^-}\sum_{\lambda\in\Lambda}\sum_{i\in I'_{\lambda}} \xi_{w_0\lambda}(u'_i u_tu'u''_i.v_{\lambda})=0\end{align*}
which is the required equation (\ref{req}). Hence we obtain equation (\ref{expressionf}) and therefore, since $\theta$ is a morphism of $ad\,U(\g)$-modules,
\begin{align*}-x.f(su')&=\sum_{\lambda\in\Lambda}\sum_{i\in I'_{\lambda}}\xi_{w_0\lambda}\Bigl(u'_iad\,x(\theta(s))u'u''_i.v_{\lambda}\Bigr)\\
&=\sum_{\lambda\in\Lambda}\sum_{i\in I'_{\lambda}}\xi_{w_0\lambda}\Bigl(u'_i(x\theta(s)-\theta(s)x)u'u''_i.v_{\lambda}\Bigr)\\
&=0\end{align*}
since moreover $\m.V'(\lambda)=\{0\}$ and $V''(\lambda)^*.\m=\{0\}$.
\end{proof}

\subsection{The dual representation of $U(\tilde\p)$ in $S(\p^-)^*$.}\label{dualS}

Recall subsection \ref{coadjointaction} and in particular the representation, denoted by $ad^*$, of $U(\tilde\p)$ in $S(\p^-)$ (see lemma \ref{coadjointaction}) and also in every $S_k(\p^-)$ ($k\in\mathbb N$).
We then can endow $S_k(\p^-)^*$ with the dual representation of $U(\tilde\p)$ given by
\begin{equation}\forall u\in U(\tilde\p),\,\forall f\in S_k(\p^-)^*,\;u.f=f\circ ad^* u^{\top},\label{actiondualbis}\end{equation} where $u^{\top}$ denotes the image of $u$ by the antipode defined similarly as in equation (\ref{antipode}). 

We have the following.
\begin{lm}
Let $k\in\mathbb N$. Then the $U(\tilde\p)$-module $S_k(\p^-)^*$ is isomorphic to the $U(\tilde\p)$-module $S_k(\tilde\p)=S_k(\p)$ when the latter is endowed with the adjoint action of $\tilde\p$ which extends by derivation the Lie bracket in $\tilde\p$.

\end{lm}

\begin{proof}

For $k=0$, the assertion is obvious. Recall (subsection \ref{genotation}) that we denote by $K$ the Killing form on $\g\times\g$. Then the vector space $\tilde\p\simeq\p$ is isomorphic to the dual space $(\p^-)^*$ through the map $$f:x\in\tilde\p\mapsto K(x,\,-)_{\mid\p^-}.$$ When $(\p^-)^*$ is endowed with the action of $U(\tilde\p)$ given by equation (\ref{actiondualbis}) and $\tilde\p$ with the adjoint action of $\tilde\p$,  the map $f$ is an isomorphism of $U(\tilde\p)$-modules. 

Indeed assume firstly that $x'\in\m$. Then for all $x\in\tilde\p$ and for all $y\in \p^-$, $x'.f(x)(y)=-K(x,\,pr_{\mathfrak r}([x',\,y]))$ by equations (\ref{eg}) and (\ref{actiondualbis}). If moreover $x\in\m$, then $$K(x,\,pr_{\mathfrak r}([x',\,y]))=0.$$ But one also has $[x',\,x]_{\tilde\p}=0$ by equation (\ref{brackettildep}).  Then $$x'.f(x)=f([x',\,x]_{\tilde\p})$$ in this case. If $x\in\mathfrak r$, then $$x'.f(x)(y)=-K(x,\,pr_{\mathfrak r}([x',\,y]))=-K(x,\,[x',\,y])$$ since $[x',\,y]=pr_{\mathfrak r}([x',\,y])+pr_{\m}([x',\,y])+pr_{\m^-}([x',\,y])$. Then by the invariance of the Killing form, we obtain that $$x'.f(x)(y)=K([x',\,x],\,y)=f([x',\,x]_{\tilde\p})(y)$$  by equation (\ref{brackettildep}).
Now if $x'\in\mathfrak r$, then the assertion follows immediatly from the invariance of the Killing form and equation (\ref{eg2}). This proves the lemma for $k=1$.

Let now $k\in\mathbb N^*$. Consider the map $f_k:S_k(\tilde\p)\longrightarrow S_k(\p^-)^*$ defined by $$f_k(x_1\cdots x_k)=K_k(x_1\cdots x_k,\,-)_{\mid S_k(\p^-)}$$ for all $x_1,\,\ldots,\,x_k\in\tilde\p$  where $K_k$ is defined as in \cite[2.7]{FJ3}, namely :
for $y_1,\,\ldots,\,y_k\in\p^-$, $$K_k(x_1\cdots x_k,\,y_1\cdots y_k)=\frac{1}{k!}\sum_{\sigma\in \mathfrak{S}_k}\prod_{i=1}^kK(x_i,\,y_{\sigma(i)})$$ where we denote by $\mathfrak{S}_k$ the group of permutations of $k$ elements.

By \cite[2.7]{FJ3} we have that $f_k$ is an isomorphism of vector spaces. It remains to show that $f_k$ is an isomorphism of $U(\tilde\p)$-modules.

Let $x_1,\,\ldots,\,x_k\in\tilde\p$, $y_1,\,\ldots,\,y_k\in\p^-$ and $x\in\tilde\p$. Then one has
\begin{align*}
x.f_k(x_1\cdots x_k)(y_1\cdots y_k)&=-K_k(x_1\cdots x_k,\,ad^*x(y_1\cdots y_k))\\
&=-\sum_{i=1}^kK_k(x_1\cdots x_k,\,y_1\cdots y_{i-1}ad^*x(y_i) y_{i+1}\cdots y_k)\\
&=-\frac{1}{k!}\sum_{i=1}^k\sum_{\sigma\in\mathfrak{S}_k}\prod_{t\neq\sigma^{-1}(i)}K(x_t,\,y_{\sigma(t)})K(x_{\sigma^{-1}(i)},\,ad^*x(y_i))\\
&=-\frac{1}{k!}\sum_{i=1}^k\sum_{\sigma\in\mathfrak{S}_k}\prod_{t\neq i}K(x_t,\,y_{\sigma(t)})K(x_{i},\,ad^*x(y_{\sigma(i)}))\\
\end{align*}

On the other hand, one has

\begin{align*}
f_k(ad_{\tilde\p}x(x_1\cdots x_k))(y_1\cdots y_k)
&=\sum_{i=1}^kf_k(x_1\cdots x_{i-1}[x,\,x_i]_{\tilde\p}x_{i+1}\cdots x_k)(y_1\cdots y_k)\\
&=\sum_{i=1}^k\frac{1}{k!}\sum_{\sigma\in\mathfrak S_k}\prod_{t\neq i}K(x_t,\,y_{\sigma(t)})K([x,\,x_i]_{\tilde\p},\,y_{\sigma(i)})\\
\end{align*}

By the case $k=1$ one obtains that, for all $1\le i\le k$, and all $\sigma\in\mathfrak{S}_k$,
$$K([x,\,x_i]_{\tilde\p},\,y_{\sigma(i)}))=-K(x_i,\,ad^*x(y_{\sigma(i)})).$$
This completes the lemma by the above.
\end{proof}

\subsection{Kostant morphism.}\label{Km}

 Recall subsection \ref{partialsym} and let $k\in\mathbb N$.

We define
$$\psi_k:\mathscr F_K^k(U(\tilde\p^-)^*)\longrightarrow S_k(\p^-)^*$$
by the following. For all $f\in\mathscr F_K^k(U(\tilde\p^-)^*)$, we set :
\begin{equation}\forall j\in\mathbb N,\; 0\le j\le k,\,\forall s\in S_j(\m^-),\,\forall s'\in S_{k-j}(\mathfrak r),\;\psi_k(f)(ss')= f(s\,\theta(s'))\label{psi}\end{equation}
that we extend by linearity, so that $\psi_k$ is a linear map.
As in \cite[6.2]{FJ2} we call $\psi_k$ the Kostant map.

\begin{prop}
Let $k\in\mathbb N$.
The kernel of the linear map $\psi_k$ is equal to $\mathscr F_K^{k+1}(U(\tilde\p^-)^*)$. Moreover $\psi_k$ is onto.
\end{prop}

\begin{proof}
It follows from the fact that $\bigoplus_{j=0}^k S_j(\m^-)\otimes \theta(S_{k-j}(\mathfrak r))$ is a complement of $U_{k-1}(\tilde\p^-)$ in $U_k(\tilde\p^-)$.
\end{proof}

Endow $U(\tilde\p^-)^*$ with the dual representation of $A$ given by equation (\ref{actiondual}). Let $k\in\mathbb N$. Then $\mathscr F_K^k(U(\tilde\p^-)^*)$ is a left $A$-module by lemma \ref{Kf} and  $S_k(\p^-)^*$ is a left $U(\tilde\p)$-module (see subsection \ref{dualS}).
\begin{lm}
Let $k\in\mathbb N$.
Then the Kostant map $\psi_k$ is a morphism from the left $A$-module $\mathscr F_K^k(U(\tilde\p^-)^*)$ to the left $U(\tilde\p)$-module $S_k(\p^-)^*$.

\end{lm}

\begin{proof}

Let $f\in\mathscr F_K^k(U(\tilde\p^-)^*)$ and $0\le j\le k$, $j\in\mathbb N$, $s\in S_j(\m^-)$ and $s'\in S_{k-j}(\mathfrak r)$.

Assume firstly that $x\in\m$.

Then by equation (\ref{psi}) \begin{align*}\psi_k(x.f)(ss')&=(x.f)(s\,\theta(s'))\\
&=-f(ad^{**}x(s\,\theta(s')))\\
&=-f(\tilde\theta(ad^*x(s))\theta(s'))\\
\end{align*}
by equations (\ref{actiondual}) and (\ref{actionA}).

Write $ad^*x(s)=\sum_{i\in I} s_iz_i$ with $s_i\in S_{j-1}(\m^-)$ and $z_i\in \mathfrak r$ for all $i\in I$, by equation (\ref{ad*}).

Then \begin{equation}\psi_k(x.f)(ss')=-\sum_{i\in I}f(s_i\theta(z_i)\theta(s'))\label{morph}\end{equation} by definition of $\tilde\theta$ (see subsection \ref{partialsym}).

On the other hand, one has by equation (\ref{actiondualbis}) :

\begin{align*} (x.\psi_k(f))(ss')&=-\psi_k(f)(ad^*x(ss'))\\
&=-\psi_k(f)(ad^*x(s)s')\end{align*}
since $ad^*x(s')=0$ by equation (\ref{eg1}).

Then by equation (\ref{psi})
\begin{equation} (x.\psi_k(f))(ss')
=-\sum_{i\in I}f(s_i\theta(z_is'))\label{morphbis}
\end{equation}

But, for all $i\in I$, $s_i\theta(z_is')=s_i\theta(z_i)\theta(s') \mod U_{k-1}(\tilde\p^-)$. Equations (\ref{morph}) and (\ref{morphbis}) imply that
$$(\psi_k(x.f)-(x.\psi_k(f)))(ss')=0$$ since $f(U_{k-1}(\tilde\p^-))=0$.

Now assume that $x\in\mathfrak r$.

\begin{align*}\psi_k(x.f)(ss')&=(x.f)(s\,\theta(s'))\\
&=-f(ad\,x(s\,\theta(s')))\\
&=-f(ad\,x(s)\theta(s')+s\,ad\,x(\theta(s')))\\
\end{align*} by equation (\ref{eg2}).

Then $$\psi_k(x.f)(ss')=-f(ad\,x(s)\theta(s')+s\,\theta(ad\,x(s')))$$ since $\theta$ is a morphism of $U(\mathfrak r)$-modules for the adjoint action.

On the other hand

\begin{align*}x.\psi_k(f)(ss')&=-\psi_k(f)(ad\,x(ss'))\\
 &=-\psi_k(f)(ad\,x(s)s'+s\,ad\,x(s'))\\
 &=-f(ad\,x(s)\theta(s')+s\,\theta(ad\,x(s')))\\
 \end{align*}
 This completes the lemma.
\end{proof}

\subsection{An isomorphism of $U(\tilde\p)$-modules}\label{isomtildep}
Recall subsections \ref{Kf} and \ref{Km}.

By proposition and lemma \ref{Km} we have that:

\begin{lm}
For all $k\in\mathbb N$, the induced morphism (still denoted by $\psi_k$) is an isomorphism of left $U(\tilde\p)$-modules from $gr_K^k(U(\tilde\p^-)^*)$ to $S_k(\p^-)^*$.

\end{lm}

\begin{proof}
We already know that the left $A$-module structure on $U(\tilde\p^-)^*$ given by equation (\ref{actiondual}) induces a left $A$-module structure on $gr_K^k(U(\tilde\p^-)^*)$ by the invariance of the Kostant filtration under the left action of $A$ (see equation (\ref{actiongradK})).
Then the induced morphism $\psi_k$ is an isomorphism from the left $A$-module $gr_K^k(U(\tilde\p^-)^*)$ to the left $U(\tilde\p)$-module $S_k(\p^-)^*$. Moreover since $S_k(\p^-)^*$ is a left $U(\tilde\p)$-module, it follows that it is the same for $gr_K^k(U(\tilde\p^-)^*)$.
Let us verify directly that $gr_K^k(U(\tilde\p^-)^*)$ is indeed a left $U(\tilde\p)$-module.  Let $x,\,x'\in\m$ and $u\in U_k(\tilde\p^-)$.  One checks that
\begin{equation}(ad^{**}x\circ ad^{**}x'-ad^{**}x'\circ ad^{**}x)(u)\in U_{k-1}(\tilde\p^-).\label{com}\end{equation}

Indeed  write $u=su'$ with  $s\in S_j(\m^-)$ and $u'\in U_{k-j}(\mathfrak r)$ for $0\le j\le k$. If $j=0$ then $ad^{**}x\circ ad^{**}x'(su')=0=ad^{**}x'\circ ad^{**}x(su')$ by equation (\ref{actionnul}). Now assume that $1\le j\le k$ and take $s=y_1\cdots y_j\in S_j(\m^-)$, with $y_i\in\m^-$ for all $1\le i\le j$.
By definition of $ad^{**}$ (see equation (\ref{actionA})), we obtain that
\begin{align*}
ad^{**}x\circ ad^{**}x'(su')&=\sum_{1\le i\neq k\le j}\prod_{t\not\in\{ i,\,k\}}y_t\theta(ad^*x(y_k))\theta(ad^*x'(y_i))u'\\
&=\sum_{1\le i\neq k\le j}\prod_{t\not\in\{ i,\,k\}}y_t\theta(ad^*x(y_k)ad^*x'(y_i))u'\mod U_{k-1}(\tilde\p^-)\\
&=\sum_{1\le i\neq k\le j}\prod_{t\not\in\{ i,\,k\}}y_t\theta(ad^*x(y_i)ad^*x'(y_k))u'\mod U_{k-1}(\tilde\p^-)\\
&=ad^{**}x'\circ ad^{**}x(su')\mod U_{k-1}(\tilde\p^-)\\
\end{align*}
since for all $a,\,b\in\mathfrak r$, one has $\theta(a)\theta(b)=\theta(ab)\mod U_1(\mathfrak r)$.

One deduces that, for all $f\in\mathscr F_K^k(U(\tilde\p^-)^*)$, for all $x,\,x'\in\m$, \begin{equation*}x.(x'.f)-x'.(x.f)\in\mathscr F_K^{k+1}(U(\tilde\p^-)^*)\end{equation*} and then
\begin{equation*}x.(x'.gr_K^k(f))=x'.(x.gr_K^k(f)).\end{equation*}
\end{proof}

Recall the notation in the proof of lemma \ref{dualS}. Let $k\in\mathbb N$ and set  $j_k=f_k^{-1}$, which is an isomorphism of $U(\tilde\p)$-modules from $S_k(\p^-)^*$ to $S_k(\tilde\p)$.  Moreover set $\psi_k^0=j_k\circ\psi_k$ and $\psi^0=\bigoplus_{k\in\mathbb N}\psi_k^0$ : this is by the above an isomorphism of $U(\tilde\p)$-modules from $gr_K(U(\tilde\p^-)^*)$ to $S(\tilde\p)$. Now as in \cite[6.6]{FJ2} set $\tilde\psi_k=\frac{1}{k!}\psi^0_k$ and $\tilde\psi=\bigoplus_{k\in\mathbb N}\tilde\psi_k$.
One deduces the following, as in \cite[6.6]{FJ2}. 
\begin{prop}
$\tilde\psi$ is an isomorphism of $U(\tilde\p)$-modules and of algebras from $gr_K(U(\tilde\p^-)^*)$ to $S(\tilde\p)$.

\end{prop}

\subsection{}\label{inj}

Denote by $gr_K(\widetilde{C}_{\mathfrak r})$ the graded algebra associated to  the induced generalized Kostant filtration on $\widetilde{C}_{\mathfrak r}$, and by $gr_K(\widetilde{C}_{\mathfrak r}^{U(\mathfrak r')})$ the graded algebra associated to  the induced generalized Kostant filtration on $\widetilde{C}_{\mathfrak r}^{U(\mathfrak r')}$. Denote by $(gr_K(U(\tilde\p^-)^*))^{U(\tilde\p')}$ the algebra of invariants in $gr_K(U(\tilde\p^-)^*)$ by the action of $U(\tilde\p')$ given by equation (\ref{actiongradK}). We have that
$$gr_K(\widetilde{C}_{\mathfrak r})\subset gr_K(U(\tilde\p^-)^*)$$
and by Proposition \ref{gra} that
$$gr_K(\widetilde{C}_{\mathfrak r}^{U(\mathfrak r')})\subset (gr_K(U(\tilde\p^-)^*))^{U(\tilde\p')}.$$
Denote also by $S(\tilde\p)^{U(\tilde\p')}$ the algebra of invariants in $S(\tilde\p)$ by the adjoint action of $U(\tilde\p')$ : this is also the algebra of semi-invariants in $S(\tilde\p)$, which we denote by $Sy(\tilde\p)$.
From Proposition \ref{isomtildep}, one deduces the following.

\begin{thm}

One has that 
$$\tilde\psi(gr_K(\widetilde{C}_{\mathfrak r}^{U(\mathfrak r')}))\subset Sy(\tilde\p).$$
\end{thm}

\subsection{}\label{cha}

Let $M$ denote a left $\h$-module such that each of its weight spaces $M_{\nu}=\{m\in M\mid \forall h\in\h,\, h.m=\nu(h)m\}$, for all $\nu\in\h^*$, is finite dimensional.
Then one may define (\cite[3.4.7]{J2}) the  formal character ${\rm ch}\, M$ of  $M$ as follows. 
$${\rm ch}\, M=\sum_{\nu\in\h^*}\dim M_{\nu}\,e^{\nu}.$$

Recall the set $E(\pi')$ in subsection \ref{polynomial},  and for all $\Gamma\in E(\pi')$, set $\delta_{\Gamma}=w_0'd_{\Gamma}-w_0d_{\Gamma}$, with $d_{\Gamma}=\sum_{\gamma\in\Gamma}\varpi_{\gamma}$. 
By Prop. \ref{filtrationC} and Theorem \ref{polynomial}, the algebra of invariants $\widetilde{C}_{\mathfrak r}^{U(\mathfrak r')}$ is a polynomial algebra in $\lvert E(\pi')\rvert$ variables, each of them having $\delta_{\Gamma}$, $\Gamma\in E(\pi')$, as an $\h$-weight. Moreover for $\g$ simple and $\pi'\subsetneq\pi$ that is, for $\p\subsetneq\g$, one has that $\delta_{\Gamma}\in P^+(\pi)\setminus\{0\}$ for all $\Gamma\in E(\pi')$ by \cite[5.4.3]{FJ2}. Then in this case ${\rm ch}\,\widetilde{C}_{\mathfrak r}^{U(\mathfrak r')}$ is well defined.

\begin{prop}
Assume, for all $\nu\in\h^*$, that the weight space $Sy(\tilde\p)_{\nu}$  is finite dimensional, so that the formal character ${\rm{ch}}\,Sy(\tilde\p)$ is well defined. Then
 $${\rm ch}\,\widetilde{C}_{\mathfrak r}^{U(\mathfrak r')}\le {\rm ch}\,Sy(\tilde\p)$$ namely
$$\prod_{\Gamma\in E(\pi')}(1-e^{\delta_{\Gamma}})^{-1}\le {\rm ch}\,Sy(\tilde\p).$$

\end{prop}
\begin{proof}
Recall that the generalized Kostant filtration $\mathscr F_K$ is decreasing, separated and that $\mathscr F_K^0(U(\tilde\p^-)^*)=U(\tilde\p^-)^*$. Then, for every finite dimensional subspace $V$ of $U(\tilde\p^-)^*$, there exists $N\in\mathbb N$ such that $V\cap\mathscr F_K^N(U(\tilde\p^-)^*)=\{0\}$. One deduces easily that  the graded vector space $gr_{K}(V)$ associated to the induced generalized Kostant filtration on $V$ is isomorphic to $V$. 
Using Theorem \ref{inj} and  a same argument as in \cite[7.1]{FJ2} completes the proof.
\end{proof}

\end{document}